\documentclass[12pt]{amsart}
\usepackage{amssymb}
\usepackage{fullpage}
\usepackage[usenames,dvipsnames]{color}

\theoremstyle{plain}
\newtheorem{thm}{Theorem}[section]
\newtheorem{prop}[thm]{Proposition}
\newtheorem{cor}[thm]{Corollary}
\newtheorem{lem}[thm]{Lemma}

\theoremstyle{definition}
\newtheorem{defn}[thm]{Definition}
\newtheorem{rem}[thm]{Remark}

\newcommand{\red}{\color{red}}

\newcommand{\bC}{\mathbb{C}}

\newcommand{\bZ}{\mathbb{Z}}

\newcommand{\eps}{\varepsilon}

\begin{document}

%%%%%%%%%%%%%%%%%%%%%%%%%%%%%%%%%%%%%%%%%%%%%%%

\title{$C^*$ exponential length of commutators unitaries in $AH$ algebras}

\author[Li, C. G.]{Chun Guang Li}
\address{School of Mathematics and Statistics,
Northeast Normal University, Changchun 130024, P. R. China}
\email{licg864@nenu.edu.cn}

\author[Li, L..]{Liangqing Li}
\address{Department of Mathematics, University of Puerto Rico, Rio Piedras, PR 00931, USA}
\email{li.liangqing@upr.edu}

\author[Ruiz, I. V.]{Iv$\acute{a}$n Vel$\acute{a}$zquez Ruiz}
\address{Department of Mathematics, University of Puerto Rico, Rio Piedras, PR 00931, USA}
\email{ivan.velazquez@upr.edu}



\date{\today}

\begin{abstract}
For each unital $C^*$-algebra $A$, we denote $cel_{CU}(A)=\sup\{cel(u):u\in CU(A)\}$,
where $cel(u)$ is the exponential length of $u$ and $CU(A)$ is the closure of the commutator subgroup of $U_0(A)$. In this paper,  we prove that $cel_{CU}(A)=2\pi$
provided that $A$ is an $AH$ algebras with  slow dimension growth whose real rank is not zero.
On the other hand, we prove that $cel_{CU}(A)\leq 2\pi$ when $A$ is an $AH$
algebra with ideal property and of no dimension growth (if we further assume $A$ is  not  of real rank zero, we have $cel_{CU}(A)= 2\pi$).
\end{abstract}

\maketitle

%%%%%%%%%%%%%%%%%%%%%%%%%%%%%%%%%
\section{Introduction}
Let $A$ be a unital $C^*$-algebra and $U(A)$ be the unitary group of $A$.
We denote by $U_0(A)$ the component of $U(A)$ containing the identity.
A unitary element $u\in U(A)$ belongs to $U_0(A)$ if and only if $u$ has
the following form
$$
u=\Pi_{j=1}^n exp(ih_j),
$$
where $n$ is a positive integer and $h_j$ is self-adjoint for every $1\leq j\leq n$.
For $u\in U_0(A)$, the exponential rank of $u$ was defined by Phillips and Ringrose
\cite{Philliringrose}, and the exponential length of $u$ was defined by Ringrose
\cite{Ringrose}. We introduce the definition of $C^*$ exponential length as follows.

%\begin{defn}
%Let $u\in U_0(A)$. If $u=\lim_{n\rightarrow\infty} u_n$,
%where $u_n=\Pi_{j=1}^kexp(ih_{n,j})$ for some $h_{n,j}=h_{n,j}^*\in A$,
%then we write
%$$
%cer(u)\leq k+\epsilon.
%$$	
%If $u=\Pi_{j=1}^kexp(ih_j)$ for some $h_i=h_i^*\in A$,
%we write
%$$
%cer(u)\leq k.
%$$
%If $cer(u)\leq k+\epsilon$ and $cer(u)\nleq k$, we write $cer(u)=k+\epsilon$.
%If $cer(u)\leq k$ and $cer(u)\nleq(k-1)+\epsilon$, we write $cer(u)=k$.
%$cer(u)$ is called the $C^*$ exponential rank of $u$.
%\end{defn}

\begin{defn}
For $u\in U_0(A)$, the $C^*$ exponential length of $u$, denoted by $cel(u)$,
is defined as
$$
cel(u)=\inf\{\sum_{j=1}^k\|h_j\|:u=\Pi_{j=1}^kexp(ih_j),h_j=h_j^*\}.
$$	
Define
$$
cel(A)=\sup\{cel(u):u\in U_0(A)\}.
$$
\end{defn}

From \cite{Ringrose}, $cel(u)$ is exactly the infimum of the length of rectifiable
paths from $u$ to ${\bf 1}_A$ in $U(A)$. Equivalently $cel(u)$ is also the infimum
of the length of smooth path from $u$ to ${\bf 1_A}$.

Exponential rank and length have been studied extensively, (see \cite{Lin2010,Phillips2,Phillips4,Phillips5,Ringrose,Zhang,Zhang1} ) and have played important roles in the classification of $C^*$-algebras (see \cite{ EGL, Gong02, homotopy,Lin2011,lin2012}).

Phillips \cite{Phillips2} proved that the exponential rank of a
unital purely infinite simple $C^*$-algebra is $1+\eps$ and
its exponential length is $\pi$. Lin \cite{Lin1} proved that
for any unital $C^*$-algebra $A$ of real rank zero, $u\in U_0(A)$
and $\eps>0$, there exists a self-adjoint element $h\in A$
with $\|h\|=\pi$ such that
$$
\|u-exp(ih)\|<\eps.
$$
This means that $cel(u)\leq \pi$. But Phillips \cite{Phillips5}
showed that when $A$ does not have real rank zero, even for some simplest example $A=M_n(C([0,1]))$, $cel(A)$ can be $\infty$.

\begin{defn}
We denote by $CU(A)$ the closure of the commutator
subgroup of $U_0(A)$ and we define
$$
cel_{CU}(A)=\sup\{cel(u):u\in CU(A)\}.
$$	
\end{defn}

In the study of classification of simple amenable
$C^*$-algebras, one often has to calculate the
exponential length for unitaties in $CU(A)$.
Pan and Wang \cite{PanWang}  constructed a simple $AI$ algebra (inductive limit of $M_n(C([0,1]))$ ) $A$ such that
$cel_{{CU}}(A)\geq 2\pi$. Applying Lin's Lemma 4.5 in \cite{lin},
the $cel_{CU}(A)$ is exactly $2\pi$.
%We are going to show that $cel_{CU}(A)=2\pi$ for a class of $AH$ algebras.

\begin{defn}
An $AH$ algebra $A$ is the $C^*$-algebra inductive limit of
a sequence $A=\lim(A_n, \phi_{n,n+1})$ with
$A_n=\oplus_{j=1}^{t_n}P_{n,j}M_{[n,j]}(C(X_{n,j}))P_{n,j}$,
where $[n,j]$ and $t_n$ are positive integers, $X_{n,j}$ are
compact metrizable spaces and $P_{n,j}\in M_{[n,j]}(C(X_{n,j}))$
are projections.
\end{defn}

By \cite{ElliottGongLi}, one can always replace the compact metrizable space $X_{n,j}$
by finite simplicial complexes, and at the same time, replace $\phi_{n,n+1} $ by injective  homomorphism (see \cite{Bi1} also).\\

{\bf In this paper,  for all $AH$ inductive limits, we will always assume that $X_{n,i}$ are connected finite simplicial complexes and all connecting maps $\phi_{n, n+1}$ are injective.}  \\

In \cite{lin}, Lin has obtained the following two main theorems (we rephrase the theorems in the language of $AH$ algebras):

$\\$ {\bf Theorem A}~(\cite{lin}, Theorem 4.6)~Suppose that $A$ is a ${\mathcal Z}$-stable simple $C^*$-algebra such that $A\otimes UHF$ is an $AH$ algebra of slow dimension growth (this class includes  all simple $AH$ algebras of no dimension growth and the Jiang-Su algebra ${\mathcal Z}$). Then $cel_{CU}(A)\leq 2\pi$.

$\\$ {\bf Theorem B}~(\cite{lin}, Theorem 5.11 and Corollary 5.12)~For any unital non elementary (i.e., not isomorphic to $M_n(\mathbb{C})$) simple $AH$ algebra $B$ of slow dimension growth, there exists a unital simple $AH$ algebra $A$ of no dimension growth
such that $(K_0(A),K_0(A)_+,K_1(A))\cong (K_0(B),K_0(B)_+,K_1(B))$ and $cel_{CU}(A)>\pi$.

\vspace{0.2in}

%{\blue{Chunguang: Your oringinal reference about slow dimension growth and no dimension growth refer to Toms-Winter is wrong. I think it is better to change the above to no dimension growth and remark it here.}}

It is proved in \cite{Winter} that for  non elementary simple $AH$ algebras, the classes of no dimension growth and slow dimension growth are the same (see \cite{EGL,Gong02,Lin2011}  also).

Our main theorem in this article is that for all (not necessary simple) $AH$ algebras $A$ with slow dimension growth, if $A$ is not of real rank zero, then $cel_{CU}(A)\geq 2\pi$. This theorem greatly generalizes and strengthens Lin's Theorem B above. If we further assume $A$ is simple, combining with Lin's Theorem A above, then  $cel_{CU}(A)=2\pi$. This gives the complete calculation of $cel_{CU}(A)$ for simple $AH$ algebras $A$ of slow dimension growth (note that for real rank zero case, it is already known by  \cite{Lin1} that $cel(A)=\pi$). We will extend such calculation of
$cel_{CU}(A)$ of simple $AH$ algebra $A$ to $AH$ algebras of no dimension growth with ideal property. We will also prove that $cel_{CU}({M_n(\mathcal Z}))\geq2\pi$ for the Jiang-Su algebra ${\mathcal Z}$. Combine with Lin's Theorem A, we have  $cel_{CU}(M_n({\mathcal Z}))=2\pi$.

In section 2, we will introduce some notations and some known results for preparation. In section 3, we will prove our main theorem. In section 4, we will deal with $AH$ algebras with ideal property. In section 5, we will calculate $cel_{CU}(M_n({\mathcal Z}))$.

%As pointed out in \cite{Bi1}, in such an indctive
%limit, one can always replace the compact metrizable space $X_{n,j}$
%by finite simplicial complexes.  	

%%%%%%%%%%%%%%%%%%%%%%
\section{Notations and some known results}

\begin{prop}[\cite{PanWang}, Lemma 2.5]\label{L:conti}
Let $u\in C([0,1])$ be defined by $u(t)=exp(i\alpha(t))$.
Then
 $$
 cel(u)=\min_{k\in \mathbb{Z}}\max_{t\in[0,1]}|\alpha(t)-2k\pi|.
 $$	
 \end{prop}

%\begin{\lem}\label{L:conti}
%Let $u\in C[0,1]$ be defined by $u(t)=exp(i\alpha(t))$.
%Then
% $$
% cel(u)=\min_{k\in \mathbb{Z}}\max_{t\in[0,1]}|\alpha(t)-2k\pi|.
% $$	
%\end{lem}

%\begin{defn}
%Let $u,v\in U(A)$, $u$ and $v$ are called approximately unitary equivalent,
%denote by $u\approx_{u.a.}v$ if there exists a sequence
%$\{w_n\}_{n=1}^\infty\subset U(A)$ such that
%$\lim_{n\rightarrow \infty}w_nvw_n^*=u$.	
%\end{defn}

%{\blue {I read and do the modification on the  pdf file you give me. But now I found out it is different from your tex file---it  has one extra lemma here, which is not in the tex file. I don't know if there are other differences.}}
%
%
%
%
%\begin{defn}
%Let $a\in M_n(\mathbb{C})$, we denote by $Sp(a)$
%the eigenvalues of $a$, counting	 multiplicity.
%\end{defn}
%
%
%\begin{prop}\label{P:approxi}
%Let $u,v\in U(M_n(C[0,1]))$. Then $u\approx_{u.a.}v$ if
%and only if $Sp(u(t))=Sp(v(t))$ for each $t\in[0,1]$.
%This result is also true for $u,v$ be self-adjoint elements.
%\end{prop}
%
%\begin{prop}[\cite{Bhatia},lines 13-18, p.71]\label{P:distance}
%Let $u.v\in U(M_n(\mathbb{C}))$. Then
%$$
%\|u-v\|\geq dist(Sp(u)-Sp(v)),
%$$
%where $dist(\{\lambda_1,\lambda_2,\dots,\lambda_n\},\{\mu_1,\mu_2,\dots,\mu_n\})=\min_{\sigma\in S_n}\max_{1\leq i\leq n}|\lambda_i-\mu_{\sigma(i)}|.$	
%\end{prop}

\begin{prop}[\cite{PanWang}, Cor. 3.5]\label{L:pertur}
Let $H_s$ be a rectifiable path in $U(M_k(C([0,1])))$. For
each $\eps>0$, there exists a piecewise smooth path $F_s$ in $U(M_k(C([0,1])))$ such that
\begin{enumerate}
	\item[(1)] $\|H_s-F_s\|_\infty<\eps$,  for all $s\in[0,1]$;
	\item[(2)] $|length_s(H_s)-length_s(F_s)|\leq \eps$;
	\item[(3)] $F_s(t)$ has no repeated eigenvalues
	for any $(s,t)\in[0,1]\times[0,1]$.
\end{enumerate}	
Moreover, if for each $t\in[0,1]$ $H_1(t)$ has no repeated
eigenvalues, then $F$ can be chosen to be such that $F_1(t)=H_1(t)$ for all $t\in[0,1]$.
\end{prop}

\begin{rem}\label{R:pertur}
In above proposition, if $H_s(0)$ and $H_s(1)$ have no repeated eigenvalues respectively, then $F$ can be chosen to satisfy that 
$F_s(0)=H_s(0)$ and $F_s(1)=H_s(1)$ for all $s\in[0,1]$.	
\end{rem}

Let $Y$ be a compact metric space. Let $P\in M_{k_1}(C(Y))$ be a projection
with $rank(P)=k\leq k_1$. For each $y$, there exists a unitary $u_y\in M_{k_1}(\mathbb{C})$(depending on $y$) such that
$$
P(y)=u^*_y
\begin{bmatrix}
1&&&&&\\
&\ddots &&&&\\
&&1&&&\\
&&&0&&\\
&&&&\ddots &\\
&&&&&0	
\end{bmatrix}u_y,
$$
where there are $k$ 1's on the diagonal. If the unitary $u_y$
can be chosen to be continuous in $y$, the projection $P$ is called
a trivial projection. It is well known that any projection $Q\in M_{k_1}(C(Y))$
is locally trivial. That is, for each $y_0\in Y$, there exists an
open set $U_{y_0}$ containing $y_0$ such that $u_y$ is continuous
on $U_{y_0}$. If $P$ is a trivial projection in $M_{k_1}(C(Y))$, we have
$$
PM_{k_1}(C(Y))P\cong M_k(C(Y)).
$$

Following the notions in \cite{PanWang}, we give the
following definitions.

\begin{defn}\label{D:equiv}
Given a metric space $(Y,d)$, we denote
$$
Y^k=\underbrace{Y\times Y\times\cdots\times Y}_k.
$$
We define a equivalent relation
on $Y^k$ as follows: two elements for $(x_1,x_1,\ldots,x_k), (y_1,y_2,\ldots,y_k)$
$\in Y^k$ are equivalent if there exists a permutation $\sigma\in S_k$ such that $x_{\sigma(i)}=y_i$
for each $1\leq i\leq k$. We let
$$
P^kY=Y^k/\sim.
$$
%Also, we define the metric on $P^kY$ as follows
%$$
%dist([x_1,x_2,\cdots,x_k],[y_1,y_2,\cdots, y_k])=\min_{\sigma\in S_k}\max_{1\leq i
%\leq k}d(x_i,y_{\sigma(i)}).
%$$
%We denote by $(R^L, d_{max})$ a metric space,
%which consists of $\{(y_1,y_2,\cdots, y_L): y_1\leq y_2\leq\cdots\leq y_L, y_i\in R,\forall 1\leq i\leq L\}$
%with $d_{max}((y_1,y_2,\cdots,y_L),(x_1,x_2,\cdots,x_L))=\max_{1\leq i\leq L}|x_i-y_i|$ for all $(x_1,x_2,\cdots, y_L)$
%and $(x_1,x_2,\cdots, x_L)$.
The metric on $P^k Y$ is defined as
$$
d([x_1,x_2,\cdots,x_k],[y_1,y_2,\cdots,y_k])=\min_{\sigma\in S^k}\max_{1\leq j\leq k}|x_j-y_{\sigma(j)}|.
$$

Let us further assume that $Y$ is compact. Let $F^kY=\mbox{Hom}(C(Y), M_k(\bC))_1$, the space of all unital homomorphisms from $C(Y)$ to $M_k(\bC)$. Then for any $\phi\in F^kY$, there are $k$ points $y_1, y_2, \cdots, y_k$ (with possible repetition) and a unitary $u\in M_k(\bC)$ such that
$$
\phi(f)=
u
\begin{bmatrix}
f(y_1)&&&\\
&f(y_2)&&\\
&&\ddots&\\
&&&f(y_k)
\end{bmatrix}u^* \hspace{0.2in} \mbox{for all}~~f\in C(Y).
$$
Define $Sp(\phi)$ to be the set $\{y_1, y_2, \cdots, y_k\}$ (counting multiplicity, see \cite{Gong02}). Considering $Sp(\phi)$ as a $k$-tuple, $(y_1, y_2, \cdots, y_k)$, it  is not uniquely determined, since the order of $k$-tuple is up to a choice; but as an element in $P^kY $, it is unique. Therefore we write $Sp(\phi)\in P^kY$.  Then
 $F^kY\ni \phi \mapsto Sp(\phi)\in P^kY$ gives a continuous map  $\Pi: F^kY \to P^kY$.
\end{defn}

\begin{prop}[\cite{PanWang}, Remark 3.9]\label{R:remark}
Let $F_s$ be a path in $U(M_k(C([0,1])))$ such that
$F_s(t)$ has no repeated eigenvalues for all $(s,t)\in[0,1]\times[0,1]$. Let $\Lambda:[0,1]\times[0,1]\rightarrow P^kS^1$ be the eigenvalue
map of $F_s(t)$, i.e., $\Lambda(s,t)=[x_1(s,t),x_2(s,t),\dots,x_k(s,t)]$, where $\{x_i(s,t)\}_{i=1}^k$ are
 eigenvalues of the matrix $F_s(t)$. There are continuous
 functions $f_1,f_2, \dots,f_k:[0,1]\times[0,1]\rightarrow S^1$ such that
 $$
 \Lambda(s,t)=[f_1(s,t),f_2(s,t),\dots,f_k(s,t)].
 $$	
For each $(s,t)\in[0,1]\times[0,1]$, there exists
a unitary $U_s(t)$ such that
$$
F_s(t)=U_s(t)diag[f_1(s,t),f_2(s,t),\dots,f_k(s,t)]U_s(t)^*.
$$
\end{prop}

Fix $1\leq i\leq n$. For each $(s,t)\in [0,1]\times [0,1]$, let $p_i(s,t)$ be the spectral projection of $F_s(t)\in M_n(\bC)$ with respect to the eigenvalue $f_i(s,t)$ (of $F_s(t)$)---this is well defined rank one projection continuously depending on $(s,t)$, since the continuous matrix value function $F_s(t)$ has distinct eigenvalues. Hence $F_s(t)=\sum_{i=1}^k f_i(s,t)p_i(s,t)$.
Since all projections in $M_n(C([0,1]\times [0,1]))$ are trivial, it is straight forward to prove that the unitary $U_s(t)$ above can be chosen to depend on $s$ and $t$ continuously.

\begin{prop}[\cite{PanWang}, Lemma 3.11] \label{L:distinct}
Let $F_s$ be a path in $U(M_n(C([0,1])))$ and $f^1_s(t),f_s^2(t),\cdots$,
$f_s^n(t)$ be continuous functions
such that
$$
F_s(t)=U_s(t)diag[f_s^1(t),f_s^2(t),\dots,f_s^n(t)]U_s(t)^*,
$$
where $U_s(t)$ are unitaries. Suppose that for any $(s,t)\in[0,1]\times[0,1]$, $f_s^i(t)\neq f_s^j(t)$ if $i\neq j$, then
$$
length_s(F_s)\geq\max_{1\leq i\leq n}\{length_s(f_s^i)\},
$$
where $f_s^i$ is regarded as a path in $U(C([0,1]))$.	
\end{prop}

\begin{defn}[\cite{Blackadar}, Definition 1.1]\label{D:eigenvar}
Let $a=a^*\in PM_n(C(X))P$, where $X$ is path connected compact metric space. For each $x\in X$, the eigenvalues of $a(x)\in P(x)M_n(\bC)P_n(x)\cong M_{rank(P)}(\bC)$ form a set of (possible repeat)   $rank(P)$ real numbers, which could be regarded as an element of $Eg(a)(x)\in P^k \mathbb{R}$, where $k=rank (P)$. On the other hand the topology on the space $\mathbb{R}$ is given by the linear order on $\mathbb{R}$ which induces a natural continuous map from $P^k\mathbb{R}$ to $\mathbb{R}^k$, by order the $k$-tuple in the increasing order; in such a  way, we identify $P^k\mathbb{R}$ as a subset of $\mathbb{R}^k$.  The map  $x\mapsto Eg(a)(x)\in P^k\mathbb{R}\subset \mathbb{R}^k$ gives $k$ continuous maps from $X$ to $\mathbb{R}$. We will call these $k$ continous maps the  eigenvalue list $E(a)$ of $a$. Namely,
the {\bf eigenvalue list} of $a$ is defined as
$$
E(a)(x)=\{h_1(x),h_2(x),\ldots,h_{k}(x)\},
$$
where $h_i(x)$ is the $i$-th lowest eigenvalue of $a(x)$,
counted with multiplicity.

The {\bf variation of the eigenvalues} of $a$ is
denoted by $EV(a)$ and is defined as
$$
EV(a)=\max_{1\leq i\leq n}\{\max_{t,s\in X}|h_i(t)-h_i(s)|\}.
$$

Here, when we use $Eg(a): X\to P^k\mathbb{R}$ and $E(a): X\to \mathbb{R}^k$, we have $E(a)=\iota\circ Eg(a)$, where $\iota: P^k\mathbb{R} \to \mathbb{R}^k$ is the natural  inclusion.

\end{defn}

\begin{rem}\label{eigen-june-16}

(1)~~ In this paper,   we will often   consider $a\in A_+$ with $\|a\|\leq 1$. Then $Sp(a)\subset [0,1]$. This element $a$ naturally defines a homomorphism $\phi: C([0,1]) \to A$ by $\phi(h)=a$, where $h:[0,1] \to [0,1]$ is identity function: $h(t)=t$. Let $A=PM_n(C(X))P$ as in \ref{D:eigenvar} . Then $E(a)$ is a map from $X$ to $[0,1]^k$
 (where $k=rank(P)$) and $Eg(a)$ is a map from $X$  to $P^k[0,1]$.

 (2)~~Let $P,Q\in M_n(C(X))$ be projections with $P<Q$. An element  $a\in (PM_n(C(X))P)_+$ can also be regarded as an element in $QM_n(C(X))Q$. The eigenvalue list $E_{PM_n(C(X))P}(a)$ of $a$ as an element in $PM_n(C(X))P$ and eigenvalue list $E_{QM_n(C(X))Q}(a)$ of $a$ as an element in $QM_n(C(X))Q$ is related in the following way. Suppose $rank(P)=k$ and $rank(Q)=l$. If $$E_{PM_n(C(X))P}(a)=\{h_1(x),h_2(x),\ldots,h_{{k}}(x)\},$$
 then  $$E_{QM_n(C(X))Q}(a)=\{\underbrace{0,\cdots,0}_{l-k},h_1(x),h_2(x),\ldots,h_{{k}}(x)\}.$$
 In particular, this positive element $a$ has the same eigenvalue variation no matter it is regarded as the element in which of the two algebras. (This is not true for general self-adjoint elements.) So when we discuss eigenvalue list or eigenvalue variation of a positive element $a$ in a corner sub algebra $PM_n(C(X))P$ of $QM_n(C(X))Q$, we don't need to specify in which algebras the calculation are made---that is, we will omit those $l-k$ constant $0$  functions from our eigenvalue list.

 (3)~~As in (2), for some (not necessarily  positive) self-adjoint elements  $a\in PM_n(C(X))P \subset QM_n(C(X))Q$, we can also ignore  in which algebra (in  the corner sub algebra $PM_n(C(X))P$ or in the algebra $QM_n(C(X))Q$), the calculation are made, when we calculate eigenvalue list and eigenvalue variation. The case is that none of the functions in the eigenvalue list of $a$ are  crossing over point $0$, that is, they are either non-positive  functions or non-negative functions. More precisely,  if
 $$E_{PM_n(C(X))P}(a)=\{h_1(x),\ldots,h_i(x), h_{i+1}(x),\cdots, h_{{k}}(x)\}$$ with $h_i(x)\leq 0\leq h_{i+1}(x)$ for all $x\in X$, then  
 $$
 E_{QM_n(C(X))Q}(a)=\{h_1(x),\ldots,h_i(x),\underbrace{0,\cdots,0}_{l-k},h_{i+1}(x),h_{{k}}(x)\}.
 $$ In this case, we will also omit those $l-k$ constant $0$  functions from our eigenvalue list for $E_{QM_n(C(X))Q}(a)$.

\end{rem}

\begin{defn}

If $A=\lim(\oplus_{j=1}^{k_n}P_{n,j}M_{[n,j]}(C(X_{n,j}))P_{n,j},\phi_{n,m})$
is a unital inductive limit system with simple limit, the following slow dimension growth condition was introduced by [BDR, Math Scand]
$$
\lim_{n\rightarrow \infty}\max_{j}\lbrace\frac{\dim(X_{n,j})+1}{rank (P_{n,j})}\rbrace=0.
$$
For general $AH$ inductive limit system, we will use the following slow dimension growth condition: for any summand $A_n^i=P_{n,i}M_{[n,i]}(C(X_{n,i}))P_{n,i}$ of a fixed $A_n$,
 $$
\lim_{m\rightarrow \infty}\max_{i,j}\lbrace\frac{\dim(X_{m,j})+1}{rank (\phi_{n,m}^{i,j}({\bf 1}_{A_n^i}))}~|~\phi_{n,m}^{i,j}({\bf 1}_{A_n^i})\not=0\rbrace=0,
 $$
where $\phi_{n,m}^{i,j}$ is the partial map of $\phi_{n,m}$ from $A_n^i$
to $A_m^j$. This notion of slow dimension growth condition is used in most literatures (see \cite{BratteliElliott}). In particular in this definition, it is automatically true that $\lim_{m\to \infty} rank (P_{m,j})=\infty$.

An inductive limit system $A=\lim(\oplus_{j=1}^{k_n}P_{n,j}M_{[n,j]}(C(X_{n,j}))P_{n,j},\phi_{n,m})$ is called of no dimension growth
if $\sup_{n,j}\dim(X_{n,j})<+\infty$. For general nonsimple inductive
limit system, no dimension growth does not imply slow dimension growth,
as it does not automatically imply that $\lim_{m\rightarrow +\infty}rank(P_{m,j})=\infty$.

 We  avoid to use the more general concept of slow dimension growth introduced by Gong \cite{Gong97} which does not imply that $\lim_{m\to \infty} rank (P_{m,j})=\infty$, since in this case our main theorem is not true. See Proposition \ref{P:finite}.

\end{defn}

\begin{prop}[\cite{BratteliElliott}, Corollary 1.3 and 1.4]\label{L:variation}
Let $A=\lim_{\rightarrow}(A_n,\phi_{n,m})$ be a $C^*$-algebra which
is the inductive limit of $C^*$-algebras $A_n$ with morphisms $\phi_{n,m}:A_n\rightarrow A_m$.
Assume that each $A_n$ has the form
$$
A_n=\oplus_{k=1}^{k_n}P_{n,k}M_{[n,k]}(C(X_{n,k}))P_{n,k}
$$
where $k_n$ are positive integers, $X_{n,k}$ are  connected compact Hausdorff spaces,
$[n,k]$ are positive integers and $M_{[n,k]}$ are the $C^*$-algebras of $[n,k]\times [n,k]$
matrices. If $A$ has slow dimension growth  (see Corollary 1.4 of \cite{BratteliElliott}) or has no dimension growth (see Corollary 1.3 of \cite{BratteliElliott}), the following are equivalent:
\begin{enumerate}
\item[(1)] $A$ has real rank zero;
\item[(2)] For any $a\in (A_n)_+$ with $\|a\|=1$ and $\varepsilon>0$,
there exists an $m$ such that every partial homomorphism
$$
\phi_{n,m}^{ij}:P_iM_{[n,i]}(C(X_{n,i}))P_i\rightarrow P_jM_{[m,j]}(C(X_{m,j}))P_j
$$
of $\phi_{n,m}$(i.e., the composition of the restriction of the
map $\phi_{n,m}$ on the $i$-th block of $A_n$ and the
quotient map from $A_m$ to the $j$-th block of $A_m$) satisfies that
$$
EV(\phi_{n,m}^{ij}(a))<\eps.
$$
\end{enumerate}
\end{prop}

In general, (1) implies (2) is always true. Predated  \cite{BratteliElliott}, it was proved in \cite{Blackadar} that
if $\dim X_{n,k}\leq 2$ for all $n$ and $k$, then (2) implies (1).

The following proposition and remark are to discuss how the eigenvalue function behaves under a homomorphism  from a single block to a single block.

\begin{prop}[\cite{ElliottGong}, Section 1.4]\label{P:singleblock}
Let $\phi: QM_{l_1}(C(X))Q \to PM_{k_1} (C(Y))P$ be a unital homomorphism, where $X$, $Y$ are connected finite simplicial complexes, and $P$, $Q$
are projections in $M_{l_1}(C(X))$ and $M_{k_1} (C(Y))$ respectively. Assume that
$rank(P)=k$, which is a multiple of $rank(Q)=l$. Then for each $y\in Y$,
$\phi(f)(y)$ only depends on the value of $f\in QM_{l_1}(C(X))Q$ at finite
many points $x_1(y),x_2(y),\ldots,x_{k/l}(y)$, where $x_i(y)$ may repeat.
In fact, if we identify $Q(x_i(y))M_{l_1}(\mathbb{C})Q(x_i(y))$ with $M_l(\mathbb{C})$, and still denote the image of $f(x_i(y))$ in $M_l(\mathbb{C})$
by $f(x_i(y))$, then there is a unitary $U_y\in M_{k_1}(C(Y))$ such that
$$
\phi(f)(y)=
P(y)U_y
\begin{bmatrix}
f(x_1(y))_{l\times l}&&&&&&\\
&f(x_2(y))_{l\times l}&&&&&\\
&&\ddots&&&&\\
&&&f(x_{k/l}(y))_{l\times l}&&&\\
&&&&0&&\\
&&&&&\ddots&\\
&&&&&&0
\end{bmatrix}U_y^*P(y).
$$
Obviously, $U_y$ depends on the identification of $Q(x_i(y))M_{l_1}(\mathbb{C})Q(x_i(y))$ and $M_{l}(\mathbb{C})$.

We denote the set (with possible multiplicity)  $\{x_1(y),x_2(y),\cdots,x_{k/l}(y)\}$ by $Sp(\phi|_y)$.

\end{prop}

\begin{rem}\label{june-15-1} One can regard $Sp(\phi|_y):=[x_1(y),x_2(y),\cdots,x_{k/l}(y)]$ as an element in $P^{k/l}X$. Then $Y\ni y \mapsto Sp(\phi|_y)\in P^{k/l}(X)$ defines a map $\phi^*: Y\to P^{k/l}X$.

Let $X, Y, Z$ be path connected compact metric spaces, and let $\alpha: Y \to P^kX$ and $\beta: Z\to P^l Y$  be two maps. Then $\alpha$  naturally induces a map $\tilde{\alpha}: P^l Y\to P^{kl} X$. We will call the the map $\tilde{\alpha}\circ\beta: Z \to P^{kl} X$ the composition of $\alpha$ and $\beta$ and denote it by $\alpha\circ \beta$.
Namely,  if $\alpha(y)=[\alpha_1(y),\alpha_2(y),\cdots, \alpha_k(y)]\in P^kX$, for all $y\in Y$ and  $\beta(z)=[\beta_1(z),\beta_2(z),\cdots, \beta_l(y)]\in P^lY$, for all $z\in Z$, then $\alpha\circ \beta$ is defined as  $$\alpha\circ\beta(z)=[\alpha_i(\beta_j(z)): 1\leq i\leq k, 1\leq j\leq l]\in P^{kl}(X) ~~~\mbox{for all}~~ z\in Z.$$

We have the following facts:\\

(a) Let $\phi: QM_{l_1}(C(X))Q \to PM_{k_1} (C(Y))P$ and $\psi: PM_{k_1}(C(Y))P \to RM_{m_1}((Z))R$ be two unital homomorphisms, then $(\psi\circ \phi)^*=\phi^*\circ \psi^*: Z\to P^{st}X$, with $\phi^*: Y\to P^sX$ and $\psi^*: Z\to P^tY$, where $s=rank(P)/rank(Q)$ and $t=rank(R)/rank(P)$.\\

(b)  Let $\phi: QM_{n}(C(X))Q \to PM_{m} (C(Y))P$($rank(Q)=k, rank(P)=kl$) be a unital homomorphism and let $f\in(QM_{n}(C(X))Q)_{s.a}$.  Use the above notation, we have
$$Eg(\phi(f))=Eg(f)\circ\phi^*: Y\to P^{kl}\mathbb{R},$$
(see \ref{D:eigenvar}). Let us write the eigenvalue list $E(f): X\to \mathbb{R}^k$ (of $f$) as $$E(f)(x)=\{h_1(x)\leq h_2(x)\leq\cdots \leq h_k(x)\}$$ with $h_i: X\to \mathbb{R}$ being continous functions for all $i$. It follows that
 $$Eg(\phi(f)(y)=[(h_1\circ \phi^*)(y),(h_2\circ \phi^*)(y),\cdots,(h_k\circ \phi^*)(y)]\in P^{kl}\mathbb{R}.$$
 For each $1\leq i\leq k$, we write the element $(h_i\circ \phi^*)(y)\in P^l\mathbb{R}$ as element \\
 $(g_{i,1}(y), g_{i,2}(y),\cdots, g_{i,l}(y))\in \mathbb{R}^l$
 in increasing order $g_{i,j}(y)\leq g_{i,j+1}(y)$. Then $g_{i,j}: Y\to [0,1]$ are continous functions with $rang(g_{i,j})\subset rang(h_i)$. Also we have
$Eg(\phi(f))(y)=[g_{i,j}(y); 1\leq i\leq k, 1\leq j\leq l]$.  (Note that, in this calculation, we did not get precise order of all the eigenfunctions $g_{i,j}$, so we use $Eg(\phi(f))(y)$ instead of $E(\phi(f))(y)$.)

\end{rem}

%%%%%%%%%%%%%%%%%%%%%%%%%%%%%%%%%%%%%%%%%%%%%%%%%%%%%%%%%%%%%%%%%%%%%%%%%%%%
\section{Main Theorem}

The following lemma and its corollary are well known and we omit the proofs.

\begin{lem}\label{L:compare}
If $u\in U_0(A)$ and $\|u-{\bf 1}_A\|<\eps<1$,
then $cel(u)\leq \frac{\pi}{2}\eps$.
\end{lem}

\begin{cor}\label{C:compare}
If $u,v\in U_0(A)$ and $\|u-v\|<\eps<1$. then $|cel(u)-cel(v)|\leq \frac{\pi}{2}\eps$.	 \end{cor}

\begin{thm}\label{T:distinct}
Suppose $u\in U_0(M_n(C([0,1])))$ have distinct eigenvalues
$\alpha_1(t),\alpha_2(t),\dots,\alpha_n(t)$, where
$\alpha_1,\alpha_2,\dots,\alpha_{n}:[0,1]\rightarrow S^1$
are continuous. Then
$$
cel(u)\geq\max_{1\leq j\leq n}cel(\alpha_j).
$$	
\end{thm}

\begin{proof}
Arbitrarily choose a path $H_s(\cdot)$ from $u$ to ${\bf 1}$.
Applying Proposition \ref{L:pertur}, for each $\eps>0$,
there exists $F_s(\cdot)$ such that
\begin{enumerate}
\item[(1)] $\|H_s-F_s\|_\infty<\eps$, for all $s\in[0,1]$;
	\item[(2)] $|length_s(H_s)-length_s(F_s)|\leq \eps$;
	\item[(3)] $F_s(t)$ has no repeated eigenvalues
	for all $(s,t)\in[0,1]\times[0,1]$.
	\item[(4)] $\Lambda(F_1(t))=[\alpha_1(t),\alpha_2(t),\dots,\alpha_n(t)]$.
\end{enumerate}	
By Proposition \ref{R:remark}, there exist continuous functions
$\beta_1(\cdot,\cdot),\beta_2(\cdot,\cdot),\dots,\beta_n(\cdot,\cdot)$ such that
$$
\Lambda(F_s(t))=[\beta_1(s,t),\beta_2(s,t)\dots,\beta_n(s,t)].
$$
Then
\begin{align*}
\Lambda(F_1(t))&=[\alpha_1(t),\alpha_2(t),\dots,\alpha_n(t)]	\\
&=[\beta_1(1,t),\beta_2(1,t),\dots,\beta_n(1,t)],
\end{align*}
and
$$
\Lambda(F_0(t))=[\beta_1(0,t),\beta_2(0,t),\dots,\beta_n(0,t)].
$$
For each $1\leq j\leq n$, we have
$$
|\beta_j(0,t)-1|\leq\max_{1\leq j\leq n}|\beta_j(0,t)-1|\leq
\|F_s(\cdot)-H_s(\cdot)\|\leq \eps.
$$
By  Lemma  \ref{L:compare}, we have
$$
cel(\beta_j(0,\cdot))\leq \frac{\pi}{2}\eps, ~~ 1\leq j\leq n.
$$
Hence,
$$
cel(\beta_j(1,\cdot))\leq cel(\beta_j(0,\cdot))+length_{s}(\beta_j(s,\cdot)), ~~ 1\leq j\leq n.
$$
By Proposition \ref{L:distinct}, we have
$$
length_{s}(F_s)\geq\max_{1\leq j\leq n}\{length_{s}(\beta_j(s,\cdot))\}.
$$
It follows that
$$
length_{s}(F_s)\geq \max_{1\leq j\leq n}\{cel(\beta_j(1,\cdot))\}-\frac{\pi}{2}\eps
=\max_{1\leq j\leq n}cel(\alpha_j)-\frac{\pi}{2}\eps.
$$
\end{proof}

Apply the above theorem, we get the following result.

\begin{thm}\label{C:increasing}
Let $u\in M_n(C([0,1]))$ with $u(t)=exp(iH(t))$, where the eigenvalue list of $H$
$$
E(H)(t)=\{h_1(t),h_2(t),\dots,h_n(t)\}
$$ 
satisfies that 
$$
-2\pi\leq \alpha\leq h_1(t)\leq h_2(t)\leq\dots\leq h_n(t)\leq
\alpha+2\pi,
$$
for some $\alpha<0$.
Then
$$
cel(u)\geq\max_{1\leq j\leq n}cel(exp(ih_j(t)))=
\max_{1\leq j\leq n}\min_{k\in \mathbb{Z}}\max_{t\in[0,1]}|h_j(t)-2k\pi|.
$$	
\end{thm}

\begin{proof}
By Corollary 1.3 in \cite{Thomsen}, without loss of generality,
we assume that
$$
H(t)=diag[h_1(t),h_2(t),\cdots,h_n(t)].
$$

We denote $a:=\min_{t\in[0,1]}h_1(t)$.

{\bf Case 1. $\alpha<a$.} For any $0<\eps<\min\{a-\alpha,1\}$,
we choose $\eps_i$ such that
$$
-\eps<\eps_1<\eps_2<\cdots<\eps_n<0.
$$
Then we have
$$
\alpha<a-\eps<h_1(t)+\eps_1<\cdots<h_n(t)+\eps_n\leq 2\pi+\alpha.
$$
We let $g_j(t)=h_j(t)+\eps_j$, $G(t)=diag[h_1(t),\cdots,h_n(t)]$ and
$v(t)=exp(iG(t))$.
Obviously, we have
\begin{align*}
\|v(t)-u(t)\|&=\|diag[exp(ig_1(t))-exp(ih_1(t)),\cdots,exp(ig_n(t))-exp(ih_n(t))]\|\\
&=\max_{1\leq j\leq n}\{\|exp(ig_j(t))-exp(ih_j(t))\|\}\\
&=\max_{1\leq j\leq n}\{|exp(i\eps_j)-1|\}\\
&=\max_{1\leq j\leq n}\{2|\sin(\frac{\eps_j}{2})|\}\\
&\leq\max_{1\leq j\leq n}\{|\eps_j|\}<\eps<1.	
\end{align*}
By Corollary \ref{C:compare}, we have
$$
|cel(v(\cdot)-cel(u(\cdot))|<\frac{\pi}{2}\eps.
$$
Notice that
$$
|h_j(t)+\eps_j-2\pi k|\geq |h_j(t)-2\pi k|-|\eps_j|>|h_j(t)-2\pi k|-\eps,
~\text{for all}~ 1\leq j\leq n, k\in \mathbb{Z}.
$$
It follows from Theorem \ref{T:distinct} and Lemma \ref{L:conti} that
\begin{align*}
cel(v(\cdot))&\geq \max_{1\leq j\leq n} cel(exp(ig_j(t)))	\\
&=\max_{1\leq j\leq n}\min_{k\in \mathbb{Z}}\max_{t\in[0,1]}|h_j(t)+\eps_j-2\pi k|\\
&\geq \max_{1\leq j\leq n}\min_{k\in \mathbb{Z}}\max_{t\in[0,1]}|h_j(t)-2\pi k|-\eps.
\end{align*}

Hence we have
$$
cel(u(\cdot))\geq\max_{1\leq j\leq n}\min_{k\in \mathbb{Z}}\max_{t\in[0,1]}|h_j(t)-2\pi k|-\eps-\frac{\pi}{2}\eps.
$$

{\bf Case 2. $\alpha=a$.} For any $0< \eps<1$, we let
$g_j(t)=\max\{h_j(t),\alpha+\eps\}$ for $1\leq j\leq n$.
We also define $G(t)=diag[g_1(t),\cdots,g_n(t)]$ and $v(t)=exp(iG(t))$.
Then we have
$$
\alpha<\alpha+\eps\leq g_1(t)\leq \cdots\leq g_n(t)\leq 2\pi+\alpha,
~\text{for all}~~ t\in[0,1].
$$
By the proof of Case 1, we have
$$
cel(v(\cdot))\geq \max_{1\leq j\leq n}\min_{k\in \mathbb{Z}}\max_{t\in[0,1]}|g_j(t)-2\pi k|.
$$
Since $|g_j(t)-h_j(t)|<\eps<1$ for all $t\in[0,1]$,
we also have $\|v(\cdot)-u(\cdot)\|<\eps<1$.
Applying Corollary \ref{C:compare}, we have
$$
|cel(exp(i g_j(\cdot)))-cel(exp(ih_j(\cdot)))|<\frac{\pi}{2}\eps, ~\text{for all}~~ 1\leq j\leq n,
$$
and
$$
|cel(v(\cdot))-cel(u(\cdot))|<\frac{\pi}{2}\eps.
$$
This means that
$$
cel(exp(ig_j(\cdot)))>cel(exp(ih_j(\cdot)))-\frac{\pi}{2}\eps, ~\text{for all}~~1\leq j\leq n,
$$
and
$$
cel(u(\cdot))>cel(v(\cdot))-\frac{\pi}{2}\eps.
$$
Hence we have
$$
cel(u(\cdot))>\max_{1\leq j\leq n}cel(exp(ih_j(\cdot)))-\pi \eps=
\max_{1\leq j\leq n}\min_{k\in \mathbb{Z}}\max_{t\in[0,1]}|h_j(t)-2\pi k|-\pi \eps.
$$
\end{proof}

\begin{cor}\label{gen-X} Let $X$ be a path connected compact metric space.
Let $P\in M_m(C(X))$ be a projection with $rank(P)=n$ and 
$u\in PM_m(C(X))P$ be with  $u(x)=exp(iH(x))$, where the eigenvalue list
of $H$ 
$$
E(H)(x)=\{h_1(x),h_2(x),\cdots, h_n(x)\}
$$
satisfies that
$$
-2\pi\leq \alpha\leq h_1(x)\leq h_2(x)\leq\dots\leq h_n(x)\leq
\alpha+2\pi,$$
for some $\alpha<0$.
Then
$$
cel(u)\geq
\max_{1\leq j\leq n}\min_{k\in \mathbb{Z}}\max_{x\in X}|h_j(x)-2k\pi|.
$$

\end{cor}

\begin{proof} For each $1\leq j\leq n$, let $x_0\in X$ and $x_1\in X$ be the minmum  and maximum points of $\{h_j(x)\}_{x\in X}$ respectively. Choose an embedding $\iota: [0,1] \to X$ satisfying  that $\iota(0)=x_0$ and $\iota(1)=x_1$.
Then
$$\min_{k\in \mathbb{Z}}\max_{x\in X}|h_j(x)-2k\pi|=\min_{k\in \mathbb{Z}}\max_{t\in [0,1]}|h_j(\iota(t))-2k\pi|.$$
Note that $cel(u)\geq cel (\iota^*(u))$, where $\iota^*: PM_m(C(X))P \to P|_{[0,1]}M_m(C([0,1]))P|_{[0,1]}\cong M_n(C([0,1]))$ is given by $\iota^*(f)(t)=f(\iota(t))$. (Note that any projection in $M_m(C([0,1]))$ is trivial, 
so $P|_{[0,1]}M_m(C([0,1]))P|_{[0,1]}\cong M_n(C([0,1]))$.)  Applying Theorem \ref{C:increasing}, we get the corollary.
\end{proof}

We shall use the following lemma and its corollary.

\begin{lem}\label{L:subsequence}
Let $f_1,f_2, \dots, f_n$  be a set
of continuous functions from $X$ to $[0,1]$, where $X$ is a connected finite simplicial complex. Let $[c,d]\subset[0,1]$ be a non degenerated sub interval. Suppose that there exists no $1\leq j \leq n$ such that $[c,d]\subset rang(f_j)$. Let $h_k(x)$ be the $k$-th lowest value of $\{f_1(x),f_2(x), \dots, f_n(x)\}$ for each $1\leq k\leq n$ and $x\in X$. Then there exists no $1\leq k\leq n$ such that $[c,d]\subset rang(h_k)$.
\end{lem}

\begin{proof}
If there exists some $1\leq k\leq n$ such that $[c,d]\subset rang(h_k)$,
we can choose $x,y\in X$ such that $h_k(x)=c$ and $h_k(y)=d$. Let
$A=\{j: f_j(x)\leq c\}$, $B=\{i: f_i(y)\geq d\}$. Since $h_k(x)=c$,
we have $|A|\geq k$. Similarly, from $h_k(y)=d$, we have $|B|\geq n-k+1$.
But $|A\cup B|\leq n$. There exists a $p\in A\cap B$. That is,
$f_p(x)\leq c$ and $f_p(y)\geq d$. Since $f_p$ is continuous, we have
$[c,d]\subset rang(f_p)$, a contradiction.
\end{proof}

\begin{cor}\label{june-17}(a) Let $\phi: PM_n(C(X))P \to QM_m(C(Y))Q$ be a unital homomorphism,and let $a\in PM_n(C(X))P $ be a self-adjoint element such that $E(a)=(h_1, h_2, \cdots, h_{rank(P)})$ and $E(\phi(a))=(f_1, f_2,\cdots, f_{rank(Q)})$ with $h_i:X\to \mathbb{R}$, and $f_k: Y\to [0, 1]$ being continuous functions. And let $[c,d]\subset \mathbb{R}$ be an interval. Then if there is a $k$ such that $[c,d]\subset rang(f_k)$, then there is an $i$, such that $[c,d]\subset rang(h_i)$. Consequently, $EV(\phi(a))\leq EV(a)$.\\

(b) Let $p_1, p_2\in PM_n(C(X))P$ be two orthogonal projections and $a_1\in p_1M_n(C(X))p_1,~~  a_2\in p_2M_n(C(X))p_2 $ be two self-adjoint elements. Then $EV(a_1+a_2)\leq max\{EV(a_1), EV(a_2)\}$.

\end{cor}

\begin{proof}Part(a): By \ref{june-15-1} (b), there are continuous functions $\{g_{i,j}: ~~ 1\leq i\leq rank(P), 1\leq j\leq rank(Q)/rank(P)\}$, with $g_{i,j}:Y\to \mathbb{R}$, such that for each $y\in Y$,
as elements in $P^{rank(Q)}\mathbb{R}$, $[f_1, f_2,\cdots, f_{rank(Q)}]=[ g_{i,j}:~~ 1\leq i\leq rank(P), 1\leq j\leq rank(Q)/rank(P)  ]$ and such that $rang(g_{i,j})\subset rang(h_i)$. Then part (a) follows from Lemma \ref{L:subsequence}.\\

Part (b) also follows from Lemma \ref{L:subsequence}.
\end{proof}

\begin{lem}\label{R:realrank1}
Let $A=\lim_{n\rightarrow \infty}(A_n,\phi_{n,n+1})$ be an
$AH$ algebra. Suppose the condition \textup{(2)} of Proposition \textup{\ref{L:variation}} does not hold for the inductive limit system (in the case of slow dimension growth or no dimension growth, this is equivalent to the condition that $A$ is not of real rank zero). There exists an interval $[c,d]\subset[0,1]$,
a positive integer $n$ and $x\in (A_n)_+$ with $\|x\|=1$ such that
for each $m\geq n$, $\phi_{n,m}(x)$ admits the following representation
$$
\phi_{n,m}(x)=\{y_k^{m}\}_{k=1}^{k_m}\in A_m=\oplus_{k=1}^{k_m}P_{m,k}M_{[m,k]}(C(X_{m,k}))P_{m,k},\eqno{(2.1)}
$$
there exist $1\leq k(m)\leq k_m$, $1\leq i(m)\leq [m,k(m)]$ such that
$$
[c,d]\subset rang(h_{i(m)}^{k(m)}),$$
where $h_i^{k(m)}(t)$ is the $i$-th lowest eigenvalue of $y^m_{k(m)}(t)$ for $1\leq i\leq [m,k(m)]$.
\end{lem}

\begin{proof}
Applying Lemma \ref{L:variation}, there exist $\eps>0$, a positive integer $n$
and $x\in (A_n)_+$ with $\|x\|=1$ such that for each $m\geq n$,
$\phi_{n,m}(x)$ admits representation (2.1), there exists
$1\leq k(m)\leq k_m$, $1\leq i(m)\leq [m,k(m)]$, $t_{i(m)},s_{i(m)}\in X_{m,k}$ such that
$$
|h_{i(m)}^{k(m)}(t_{i(m)})-h_{i(m)}^{k(m)}(s_{i(m)})|\geq \eps,
$$
where $h_i^{k(m)}(t)$ is the $i$-th lowest eigenvalue of $y^m_{k(m)}(t)$ for $1\leq i\leq [m,k(m)]$.  For $m\geq n$, we denote by $I^{k(m)}_{i(m)}$ the
closed interval with end points $h_{i(m)}^{k(m)}(s_{i(m)})$ and $h_{i(m)}^{k(m)}(t_{i(m)})$. We also denote by $\overline{I}^{k(m)}_{i(m)}$ the closed
interval with the same middle point as $I^{k(m)}_{i(m)}$ so that
$$
|\overline{I}^{k(m)}_{i(m)}|=\frac{1}{2}|I^{k(m)}_{i(m)}|\geq\frac{\eps}{2}.
$$

Choose a positive integer $N$ such that $\frac{2}{N}<\eps$.
We denote $a_p=\frac{p}{N}$ for $0\leq p\leq N$.
Since $|\overline{I}^{k(m)}_{i(m)}|\geq \frac{1}{2}\eps$
and $\overline{I}^{k(m)}_{i(m)}\subset[0,1]$ for all $m\geq n$,
there exist a $0\leq p\leq N$ and a subsequence $m_j$ such that
$$
a_p\in \overline{I}^{k(m_j)}_{i(m_j)}, ~\text{for all}~ j\geq1.
$$
Denote $I=[a_p-\frac{\eps}{4},a_p]$ and $J=[a_p,a_p+\frac{\eps}{4}]$, then $I\subset I_{i(m_j)}^{k(m_j)}$ or $J\subset I_{i(m_j)}^{k(m_j)}$
for each $j\geq1$. Without loss of generality, we assume that
$I\subset I_{i(m_j)}^{k(m_j)}$ for each $j\geq1$. Otherwise, we shall
choose a subsequence of $\{m_j\}_{j=1}^\infty$.

We have proved that the conclusion holds for $m_j$ for each $j\geq1$.
For $m\geq n$, there exists $j\geq1$ such that $m_{j-1}<m\leq m_{j}(m_0=n)$.
We consider
$$
\phi_{m,m_j}^{l,k(m_j)}:P_{m,l}M_{[m,l]}(C(X_{m,l}))P_{m,l}\rightarrow
P_{m_j,k(m_j)}M_{[m_j,k(m_j)]}(C(X_{m_j,k(m_j)}))P_{m_j,k(m_j)},
$$
the homomorphism which is composition of the restriction of $\phi_{m,m_j}$
on the $l$-th block of $A_m$ and the quotient map from $A_{m_j}$ to
the $k(m_j)$-th block of $A_{m_j}$.

We claim that there exist $1\leq k(m) \leq k_m$ and $1\leq i(m)\leq [m,k(m)]$ such that
$$
I\subset rang(h_{i(m)}^{k(m)}),
$$
where $h_i^{k(m)}(t)$ is the $i$-th lowest eigenvalue of $y^m_{k(m)}(t)$.
Otherwise, for each $1\leq k\leq k_m$ and $1\leq i\leq [m,k]$,
$rang(h_i^k)$ does not contain the interval $I$. By Corollary \ref{june-17},
we conclude that there exists no $1\leq k\leq k_{m_j}$
and $1\leq i\leq [m_j,k]$ such that
$$
I\subset rang(g_{i}^{k})
$$
where $g_i^{k}(t)$ is the $i$-th lowest eigenvalue of $y^{m_j}_{k}(t)$. A contradiction.
\end{proof}

Let $\pi_j: A_m\to A_m^j$ be the projection map to the $j$-th block. In the proof of the following theorem and the rest of the paper, let us denote $\pi_j\circ\phi_{n,m}$ by $\phi_{n,m}^{-,j}$ which is the homomorphism from $A_n$ to $A_m^j$.

\begin{thm}\label{T:Main}
Let A be a unital $AH$ algebra of slow dimension growth condition which is not of real rank zero.
Then
$$
cel_{CU}(A)\geq2\pi.
$$
\end{thm}

\begin{proof}
Since $A$ is not real rank zero, and each $A_k$ is unital, there exists $k_0$
such that for all $k\geq k_0$, $\phi_{k,\infty}(1_{A_k})A\phi_{k,\infty}(1_{A_k})$ is not real rank zero.
For any $\eps>0$, choose an integer $L$ such that $\frac{2\pi}{L}<\eps$,
since $A$ has slow dimension growth condition, there exist $n$ and a full projection
$p\in A_n$ such that
$$
L[p]<1_{A_n}<m[p],
$$
for some $m$. Also, we can assume that $\phi_{n,\infty}(1_{A_n})A\phi_{n,\infty}(1_{A_n})$ is not
real rank zero. Then $\phi_{n,\infty}(p) A\phi_{n,\infty}(p)$ is stably isomorphic to $A$ and hence \\
$\phi_{n,\infty}(p) A\phi_{n,\infty}(p)=\lim(\phi_{n,m}(p)A_m \phi_{n,m}(p), \phi_{m,m'})$ (by abusing the notation, we still use the  $\phi_{m,m'}$ to denote the restriction of the map to the conner subalgebra $\phi_{n,m}(p)A_m \phi_{n,m}(p)$)  is not of real rank zero. By Lemma \ref{R:realrank1},
there exist an interval $[c,d]\subset[0,1]$, an integer $n_1\geq n$ and a positive element $x\in (\phi_{n,n_1}(p)A_{n_1}\phi_{n,n_1}(p))_+$
with $\|x\|=1$ such that for every $m\geq n_1$, $\widetilde{\phi}_{n_1,m}(x)$ has the following representation
%
%{\blue{Attension: an $n$-tuple with several components can not be written as $diag[y_1^m,y_2^m,\dots,y_{k_m}^m]$, your original notation means one matrix with diagonal entries,but here are $k_m$ matrices.}}
$$
\widetilde{\phi}_{n_1,m}(x)=(y_1^m,y_2^m,\dots,y_{k_m}^m)\in \bigoplus_{i=1}^{k_m}\phi_{n,m}^{-,i}(p) A_m^i\phi_{n,m}^{-,i}(p),
$$
where $\widetilde{\phi}_{n_1,m}=\phi_{n_1,m}|_{\phi_{n,n_1}(p)A_{n_1}\phi_{n,n_1}(p)}$.
There exist $1\leq k(m)\leq k_m$, $1\leq i(m)\leq [m,k(m)]$ such that
$$
[c,d]\subset rang(h_{i(m)}^{k(m)}),
$$
 where $h_i^{k(m)}(t)$ is the $i$-th lowest eigenvalue of $y_{k(m)}^m(t)$ for $1\leq i\leq [m,k(m)]$.

Since $p$ is a full projection in $A_n$ and $L[p]<1_{A_n}$, there exists a set of mutually orthogonal rank one
 projections $p_1,p_2\dots, p_L\in A_n$ such that $p_i\sim p_j\sim p$ and $\sum_{i=1}^Lp_i<{\bf 1}_{A_n}$.
We let $q=\sum_{i=1}^Lp_i$. It is easy to see that $qA_nq$ and $M_L(pA_np)$ are isomorphic.
This means that  $M_L(pA_np)\subset A_n$ and hence $M_L(\phi_{n,n_1}(p)A_{n_1}\phi_{n,n_1}(p))\subset A_{n_1}$.

We define a continuous function on $[0,1]$ as follows:
\begin{equation*}
\chi(t)=
\begin{cases}
0,\quad\quad\quad\quad\quad\quad\quad    t\in[0,c]\\
\frac{1}{(d-c)}(x-c),\quad\quad\quad t\in[c,d]\\
1,\quad\quad\quad\quad\quad\quad\quad   t\in[d,1]
\end{cases}.\eqno(2.2)
\end{equation*}
Further, we define two continuous functions $\chi_1:[0,1]\rightarrow [0,\frac{1}{L}]$
and $\chi_2:[0,1]\rightarrow [-1+\frac{1}{L}, 0]$ as follows
$$
\chi_1(t)=\frac{1}{L} t,
$$
and
$$
\chi_2(t)=(-1+\frac{1}{L})t.
$$
We let
$$
h=
\begin{bmatrix}
e^{2\pi i\chi_2\circ\chi(x)}&0&\cdots&0\\
0&e^{2\pi i\chi_1\circ\chi(x)}&\cdots&0\\
\vdots&\vdots&\ddots&\vdots\\
0&0&0&e^{2\pi i\chi_1\circ\chi(x)}
\end{bmatrix}_{L\times L}\in  \phi_{n,n_1}(q)A_{n_1}\phi_{n,n_1}(q),
$$
where $\chi(x)\in  \phi_{n,n_1}(p)A_{n_1}\phi_{n,n_1}(p)$ and $\chi_i\circ\chi(x)\in  \phi_{n,n_1}(p)A_{n_1}\phi_{n,n_1}(p)$ are functional calculus of self-adjoint element $x$, also we identify $\phi_{n,n_1}(q)A_{n_1}\phi_{n,n_1}(q)\cong M_L( \phi_{n,n_1}(p)A_{n_1}\phi_{n,n_1}(p))$.

Let $u=h\oplus({\bf 1}_{A_{n_1}}-\phi_{n,n_1}(q))$. It is easy to check that $det(u(z))=1$ for all $z\in Sp(A_{n_1})$,
furthermore we have $u\in U_0(A_{n_1})$. It follows from \cite{Phillips96} that
$u\in CU(A_{n_1})$.

We shall show that $cel(\phi_{n_1,m}(u))\geq 2\pi(1-\eps)$, for all $m\geq n_1$.
For a fixed $m\geq n_1$, we have
 $
 \phi_{n_1,m}(h)=exp(2\pi iH)$,
 where
 $$
 H=
 \begin{bmatrix}
\chi_2\circ\chi(\widetilde{\phi}_{n_1,m}(x))&0&\cdots&0\\
0&\chi_1\circ\chi(\widetilde{\phi}_{n_1,m}(x))&\cdots&0\\
\vdots&\vdots&\ddots&\vdots\\
0&0&0&\chi_1\circ\chi(\widetilde{\phi}_{n_1,m}(x))
\end{bmatrix}_{L\times L},
 $$
and $\widetilde{\phi}_{n_1,m}=\phi|_{\phi_{n,n_1}(p)A_{n_1}\phi_{n,n_1}(p)}$.
It follows that $\widetilde{\phi}_{n_1,m}(x)\in \phi_{n,m}(p)A_{m}\phi_{n,m}(p)$
and hence $H\in M_L(\phi_{n,m}(p)A_{m}\phi_{n,m}(p))=\phi_{n,m}(q)A_{m}\phi_{n,m}(q)\subset A_m$.
%{\blue{It is really confuse to use $\phi_{n_1,m}$ for different meaning, if $\chi_2\cdot\chi(\phi_{n_1,m}(x))$ is already element of $A_m$, how could above matrix of such element is still an element of $A_m$ (not $M_L(A_m)$). I think $\chi_2\cdot\chi(\phi_{n_1,m}(x))$ really mean $\chi_2\cdot\chi(\phi_{n_1,m}|_{ \phi_{n,n_1}(p)A_{n_1}\phi_{n,n_1}(p)}(x))$ which is in  $\phi_{n,m}(p)A_{m}\phi_{n,m}(p)$. Make the correction carefully, by introduce proper notations.}}

%Let us denote $\pi_i\cdot\phi_{n,m}$ by $\phi_{n,m}^{-,j}$ which
%is the homomorphism from $A_n$ to $A_m^j$(the $j$-th block of $A_m$).
Note that $\widetilde{\phi}_{n_1,m}(x)=(y_1^m,y_2^m,\cdots,y_{k_m}^m)$
with each $y_j^m\in \phi_{n,m}^{-,j}(p)A_m^j\phi_{n,m}^{-,j}(p)$.
There exists $1\leq k(m)\leq k_m$,
$1\leq i(m)\leq [m,k(m)]$ such that
$$
[c,d]\subset rang(h_{i(m)}^{k(m)}),
$$
 where $h_i^{k(m)}(t)$ is the $i$-th lowest eigenvalue of $y_{k(m)}^m(t)$ for $1\leq i\leq [m,k(m)]$.
 we have
 $$
 \chi_2\circ\chi(\widetilde{\phi}_{n_1,m}(x))=(\chi_2\circ\chi(y_1^m),\chi_2\circ\chi(y_2^m),\cdots,\chi_2\circ\chi(y_{k_m}^m))
\in \bigoplus_{i=1}^{k_m}\phi_{n,m}^{-,i}(p) A_m^i\phi_{n,m}^{-,i}(p)$$
 and
  $$
 \chi_1\circ\chi(\widetilde{\phi}_{n_1,m}(x))=(\chi_1\circ\chi(y_1^m),
 \chi_1\circ\chi(y_2^m),\cdots,\chi_1\circ\chi(y_{k_m}^m))
 \in \bigoplus_{i=1}^{k_m}\phi_{n,m}^{-,i}(p) A_m^i\phi_{n,m}^{-,i}(p).$$

Write $H=(H_1, H_2, \cdots, H_{k_m})\in \bigoplus_{i=1}^{k_m}\phi_{n,m}^{-,i}(q) A_m^i\phi_{n,m}^{-,i}(q).$ It follows that
\begin{align*}
 H_i&=\begin{bmatrix}
\chi_2\circ\chi(y_i^m)&0&\cdots&0\\
0&\chi_1\circ\chi(y_i^m)]&\cdots&0\\
\vdots&\vdots&\ddots&\vdots\\
0&0&0&\chi_1\circ\chi(y_i^m)
\end{bmatrix}_{L\times L}.\\
%&\cong diag[\begin{bmatrix}
%\chi_2\cdot\chi(y_i^m)&0&\cdots&0\\
%0&\chi_1\cdot\chi(y_i^m)&\cdots&0\\
%\vdots&\vdots&\ddots&\vdots\\
%0&0&0&\chi_1\cdot\chi(y_i^m)
%\end{bmatrix}_{L\times L}
%]_{i=1}^{k_m}\\
%&=diag[H_1(x),H_2(x),\dots,H_{k_m}(x)].
\end{align*}

%and $H(x)\in \phi_{n,m}(q)A_m\phi_{n,m}(q)$.

%{\blue{once you use the correct notation, you know the above $H$ is not a self adjoint element in $A_m$, it is in $\phi_{n,m}(q)A_m\phi_{n,m}(q)$.}}

This means that
$$
\phi_{n_1,m}(h)=exp(2\pi i H)=(exp(2\pi iH_1),exp(2\pi iH_2),\dots,exp(2\pi iH_{k_m})),
$$
%{\blue{(again, the above is an unitary in  $\phi_{n,m}(q)A_m\phi_{n,m}(q)$, you should direct sum ${\bf 1}_{A_m}-\phi_{n,m}(q)$ to get a unitary in $A_m$)}}
and hence
$$
cel(\phi_{n_1,m}(h))\geq cel(exp(2\pi i H_k(x))),~~\mbox{for all}~~ 1\leq k\leq k_m.
$$
In particular, we have $cel(\phi_{n_1,m}(h))\geq cel(exp(2\pi iH_{k(m)}))$,
and\\ $cel(\phi_{n_1,m}(u))\geq cel(exp(2\pi i H_{k_m})\oplus ({\bf 1}_{A_{m}^{k(m)}}-\phi_{n,m}^{-,k(m)}(q)))$.
Furthermore, the eigenvalue list of $H_{k(m)}$ satisfies that
\begin{align*}
-1+\frac{1}{L}&\leq\chi_2\circ\chi\circ h_{[m,k]}^{k(m)}\leq \chi_2\circ\chi\circ h_{[m,k]-1}^{k(m)}\leq \cdots\leq \chi_2\circ\chi\circ h_1^{k(m)}\\
&\leq
\underbrace{\chi_1\circ\chi\circ h_1^{k(m)}\leq\chi_1\circ\chi\circ h_1^{k(m)}\leq \cdots\leq\chi_1\circ\chi\circ h_1^{k(m)}}_{L-1}\\
&\leq
\underbrace{\chi_1\circ\chi\circ h_2^{k(m)}\leq\chi_1\circ\chi\circ h_2^{k(m)}\leq \cdots\leq\chi_1\circ\chi\circ h_2^{k(m)}}_{L-1}\\
&\leq\cdots\\
&\leq
\underbrace{\chi_1\circ\chi\circ h_{[m,k]}^{k(m)}\leq\chi_1\circ\chi\circ h_{[m,k]}^{k(m)}\leq \cdots\leq\chi_1\circ\chi\circ h_{[m,k]}^{k(m)}}_{L-1}\leq \frac{1}{L}.
\end{align*}
That is, $\phi_{n_1,m}^{-,k(m)}(h)=exp(2\pi i H_{k(m)})$ satisfies the condition of Corollary \ref{gen-X}. Applying (3) of Remark \ref{eigen-june-16}, we know that
$\phi_{n_1,m}^{-,k(m)}(u)=exp(2\pi i H_{k(m)})\oplus ({\bf 1}_{A_{m}^{k(m)}}-\phi_{n,m}^{-,k(m)}(q)))$ also satisfies the condition of Corollary \ref{gen-X}. By the corollary, we have
$$
\phi_{n_1,m}^{-,k(m)}(u)\geq \min_{p\in \mathbb{Z}}\max_{y\in X_{m,k(m)}}2\pi|\chi_2\circ\chi\circ h_{i(m)}^{k(m)}(y)-p|.
$$
Noting that $[c,d]\subset rang(h_{i(m)}^{k(m)})\subset[0,1]$, by the definitions of $\chi$
and $\chi_2$, we have
$$
rang(\chi_2\circ\chi\circ h_{i(m)}^{k(m)})=[-1+\frac{1}{L},0]
$$
and hence
$$
\min_{p\in \mathbb{Z}}\max_{y\in X_{m,k(m)}}2\pi|\chi_2\circ\chi\circ h_{i(m)}^{k(m)}(y)-p|=(1-\frac{1}{L})2\pi\geq 2\pi-\eps.
$$
\end{proof}

\begin{rem}\label{july-23-2018}Evidently, our proof also works for the case of no dimension growth provided that $\lim_{n\to \infty} \min_i\{rank({\bf 1}_{A_n^i})\}=\infty$. ,	
\end{rem}

For all $k\geq 1$ and $u\in CU(M_k(C([0,1])))$, Lin \cite{lin} prove that
$cel(u)\leq 2\pi$, in fact, we have the following proposition.
This proposition also tell us, we can not replace the slow dimension growth
condition by Gong's slow dimension growth condition, which does not imply that $\lim_{n\to \infty} \min_i\{rank({\bf 1}_{A_n^i})\}=\infty$..

\begin{prop}\label{P:finite}
%For all $k\geq 1$ and $u\in
$cel_{CU}(M_k(C([0,1])))=\frac{k-1}{k}2\pi$.
%Lin \cite{lin} prove that
%$length(u)\leq 2\pi$, in fact, we have $length(u)\leq \frac{k-1}{k}2\pi$.
%This generalizes Lin's result. Further, This means that the concusion of
%Theorem \ref{T:Main} can not hold in such algebras.
\end{prop}

\begin{proof} From the construction in \cite{PanWang}, we know that $cel_{CU}(M_k(C([0,1])))\geq \frac{k-1}{k}2\pi$.
the following proof of $cel_{CU}(M_k(C([0,1])))\leq \frac{k-1}{k}2\pi$ is inspired by Section 3 of \cite{GLN1}(see also the proof of
Lemma 4.2 in \cite{lin}). Let $u\in CU(M_k(C([0,1])))$ and $\eps>0$,
using the proof of Lemma 4.2 in \cite{lin}, we can find $v\in CU(M_k(C([0,1])))$
which satisfies the following conditions.
\begin{enumerate}
\item[(1)] $v(t)=\sum_{j=1}^kexp(2\pi i h_j(t))p_j(t)$,
where $h_j(t)\in C([0,1])_{s.a}$ and $\{p_1,p_2,\dots, p_k\}$
is a set of mutually orthogonal rank one projections,
\item[(2)] $\sum_{j=1}^kh_j(t)=0$ for all $t\in[0,1]$, this ensures
that $det(v(t))=0$ for all $t\in[0,1]$ and hence $v\in CU(M_k(C([0,1])))$,
\item[(3)] $h_j(t)- h_l(t)\notin \mathbb{Z}$ for any $t\in[0,1]$ when $j\neq l$,
this ensures that $v(t)$ has distinct eigenvalues, further,
$0<\max_{1\leq j\leq k}h_j(t)-\min_{1\leq j\leq k}h_j(t)<1$
for all $t\in [0,1]$.
\item[(4)] $|h_j(t)|<1$ for all $t\in[0,1]$ and $1\leq j\leq k$.
\item[(5)] $\|u-v\|<\eps$.
\end{enumerate}
We shall show that
$$
\|h_j\|<\frac{k-1}{k},
~\text{for all}~ 1\leq j\leq k.
$$
Since $\{h_j\}_{j=1}^k$ satisfy condition (3),
without loss of generality, we can assume that
$$
h_1(t)>h_2(t)>\dots>h_k(t), h_1(t)-h_k(t)<1, ~\text{for all}~ t\in[0,1].
$$
For fixed $1\leq k_0\leq k$ and $t\in[0,1]$, we have
\begin{align*}
0&=h_1(t)+h_2(t)+\dots+h_{k_0}(t)+\dots+h_k(t)\\
&>k_0h_{k_0}(t)+(k-k_0)h_k(t)\\
&>k_0h_{k_0}(t)+(k-k_0)(h_1(t)-1)\\
&>k_0h_{k_0}(t)+(k-k_0)(h_{k_0}(t)-1)\\
&=kh_{k_0}(t)-k+k_0,	
\end{align*}
hence
$$
h_{k_0}(t)<\frac{k-k_0}{k}.
$$
On the other hand, we have
\begin{align*}
0&=h_1(t)+h_2(t)+\dots+h_{k_0}(t)+\dots+h_k(t)\\
&<(k_0-1)h_1(t)+(k-k_0+1)h_{k_0}(t)\\
&<(k_0-1)(1+h_k(t))+(k-k_0+1)h_{k_0}(t)\\
&<(k_0-1)(1+h_{k_0}(t))+(k-k_0+1)h_{k_0}(t)\\
&<kh_{k_0}(t)+k_0-1,	
\end{align*}
hence
$$
h_{k_0}(t)>-\frac{k_0-1}{k}.
$$
It follows that
$$
\|h_{k_0}\|<\frac{k-1}{k}.
$$
We let
$$
v_s(t)=\sum_{j=1}^kexp(2\pi ish_j(t))p_j(t), ~\text{for all}~ s\in[0,1], t\in[0,1].
$$
Then $v_s$ is a path with $v_0(t)=v(t)$ and $v_1(t)={\bf 1}$.
Further, we have
\begin{align*}
length_s(v_s)&=\int_{0}^1\|\frac{dv_s}{ds}\|ds\\
&=\int_0^1\|\sum_{j=1}^k2\pi ih_j(t)exp(2\pi ish_j(t))p_j(t)\|ds\\
&=2\pi\int_0^1\max_{1\leq j\leq k}\|h_j\|ds\\
&<2\pi\frac{k-1}{k}.
\end{align*}
By (4) and Corollary \ref{C:compare}, it follows that
$$
cel(u)\leq cel(v)+\frac{\pi}{2}\eps\leq length_s(v_s)+\frac{\pi}{2}\eps<2\pi\frac{k-1}{k}+\frac{\pi}{2}\eps.
$$
As $\eps$ goes to zero, we have $cel(u)\leq 2\pi\frac{k-1}{k}$.
\end{proof}

%%%%%%%%%%%%%%%%%%%

\section{exponential length in $AH$ algebras with ideal property}

\begin{defn}
We say a $C^*$-algebra $A$ has the ideal property, if each closed two
sided ideal of $A$ is generated by the projections inside the ideal.
\end{defn}

Evidently, all simple $AH$ algebras and all real rank zero $C^*$-algebras
have ideal property. In this part, we shall show that $cel_{CU}(A)\leq2\pi$
for each $AH$ algebras with ideal property which is of no dimension growth.

As in \cite{ElliottGong}, we denote by $T_{II,k}$ the 2-dimensional connected
simplicial complex with $H^1(T_{II,k})=0$ and $H^2(T_{II,k})=\mathbb{Z}/{k\mathbb{Z}}$,
and we denote by $I_k$ the subalgebra of $M_k(C([0,1]))=C([0,1],M_k(\mathbb{C}))$ consisting of functions
$f$ with property $f(0)\in \mathbb{C}1_k$ and $f(1)\in \mathbb{C}1_k$. $I_k$
is called Elliott dimension drop interval algebra. As in \cite{Gongjiangli}, we denote by $\mathcal{HD}$
the class of algebras of direct sums of building blocks of forms $M_l(I_k)$ and
$PM_n(C(X))P$, with $X$ being one of the spaces $\{pt\}$, $[0,1]$, $S^1$ and $T_{II,k}$, and with $P\in M_n(C(X))$ being a projection. We will call a $C^*$-algebra an $\mathcal{AHD}$ algebra, if it is an inductive limit of algebras in $\mathcal{HD}$.
In \cite{GJLP2,GJLP1}, it is proved that all $AH$ algebras with ideal property of
no dimension growth are $\mathcal{AHD}$ algebras.

%\begin{lem}\label{L:codimension}
%Let $Z=\{u\in U(PM_n(C(T_{II,k}))P):~\text{u has repeated eigenvalues}\}$.
%Then $Z$ is a union of finitely many submanifolds of $U(PM_n(C(T_{II,k}))P)$,
%all of codimension at least three.
%\end{lem}
%{\red{The lemma does not make sense since $U(PM_n(C(T_{II,k}))P)$ is not a manifold. I guess you mean $U_{rank(P)}$ is a manifold and the set with repeat eigenvalues a union of finitely many submanifolds with codimension at least three. Nothing to do with
%$T_{II,k}$.}}

\begin{lem}[\cite{PanWang}, Corollary 3.2]\label{L:codimension}
Let $Z=\{u\in U(M_n(\mathbb{C})):~~\textup{$u$ has repeated eigenvalues}\}$.
Then $Z$ is the union of finitely many submanifolds of $U(M_n(\mathbb{C}))$,
all of codimension at least three.
\end{lem}

%Let $u\in U(PM_n(C(T_{II,k}))P)$, where $P\in M_n(C(T_{II,k}))$ is a projection with $rank(P)=k\leq n$.
%For each $y\in T_{II,k}$, there exists $0\leq\alpha_1(y)\leq\alpha_2(y)\leq\cdots,\leq\alpha_k(y)\leq2\pi$ {\blue{(This is  not true, in  particular, it is impossible to chosen all of them in the range of $[0,2\pi]$, but I don't think we need it.)}}
%and a unitary $V_y$ such that
%$$
%u(y)=P(y)V_y
%\begin{bmatrix}
%e^{\alpha_1(y)}&&&&&\\
%&\ddots&&&&\\
%&&e^{\alpha_k(y)}&&&\\
%&&&0&&\\
%&&&&\ddots&\\
%&&&&&0
%\end{bmatrix}V_y^*P(y).
%$$
%Write $Sp(u(y)):=\{\alpha_1(y),\alpha_2(y),\cdots,\alpha_k(y)\}$,
%where we count multiplicity. Furthermore, for each $1\leq i\leq k$,
%$\alpha_i(y)$ can be chosen to be continuous{\blue{(this is what we need to do in Lemma 3.6 below)}}.

Noting that $\dim(T_{II,k})=2$,
by Lemma \ref{L:codimension} and a standard transversal argument, we can get the following result.

\begin{lem}\label{L:transversal}
Let $u\in U(PM_n(C(T_{II,k}))P)$, where $P$ is a projection in $M_n(C(T_{II,k}))$. For any $\eps>0$,
there exists $v\in U(PM_n(C(T_{II,k}))P)$ such that
\begin{enumerate}
\item[(1)] $\|u-v\|\leq\eps$;
\item[(2)] and $Sp(v(y))=\{\beta_1(y),\beta_2(y),\ldots\beta_k(y)\}$, where
$k=rank(P)$ and $\beta_i(y)\neq\beta_j(y)$ for all $i\neq j$ and $y\in T_{II,k}$.
\end{enumerate}
\end{lem}

%\begin{rem}
%In the above lemma, we only get $Sp(v(y)$ has distinct point, not continuous function $\beta_j(y)$.
%\end{rem}

\vspace{0.2in}

Recall that   $F^k S^1=Hom(C(S^1), M_k(\bC))_1$  and $\Pi: F^k S^1\to P^k S^1$ be as  Defined in  \ref{D:equiv}. Let $\mbox{\r{F}}^k S^1$ be the set of homomorphism with distinct spectrum and  $\mbox{\r{P}}^k S^1 =\Pi( \mbox{\r{F}}^k S^1 )$.

\begin{lem}\label{L:pi_1}
$\pi_1(\mbox{\r{P}}^k S^1)=\bZ$ is torsion free.

\end{lem}

\begin{proof} Note that $F^k S^1$ is homeomorphic to $U_k(\bC)=U(k)$, and $\mbox{\r{F}}^k S^1$ corresponds to the set of all unitaries  $u\in U(k)$ with distinct spectrum, which is a union of finitely many sub-manifolds of  $U(k)$ of codimensions at least three.  Hence $\pi_1(\mbox{\r{F}}^k S^1)=\pi_1(U(k))=\bZ$.

Consider the fibration map $\Pi|_{\mbox{\r{F}}^k S^1}: \mbox{\r{F}}^k S^1 \to \mbox{\r{P}}^k S^1$, whose fibre is the simply connected flag manifold $U(k)/\underbrace{U(1)\times U(1)\times\cdots U(1)}_k$, we get the desired result.
\end{proof}

%Let $Y$ be a simplicial complex and $P^k(Y)$ be as in Definition \ref{D:equiv}.
%Let $\pi:Y^k\rightarrow P^kY$ be the quotient map. Let $X\subset Y^k$ be the set consisting
%of those elements $(x_1,x_2,\ldots,x_k)$ with $x_i\neq x_j$ for $i\neq j$. Then the restriction
%of $\pi$ to $X$ is a covering map.
%With above notion, when $Y=S^1$, by a similar argument of the proof of
%Corollary 6.11 in \cite{Dardalat}, we have $\pi_1(\pi X)$ is torsion free.
%Noting that $\pi_1(T_{II,k})=\mathbb{Z}_k$, then for any continuous
%map $f: T_{II,k}\rightarrow \pi X$, we have $f_*(\pi_1(T_{II,k}))=\{0\}$.

%By the lifting lemma Proposition 1.33 in \cite{Hatcher}, we have the following result.
%Maybe the following result is well known, but we didn't find a reference.
%We give a proof.
%

\begin{lem}\label{L:lifting}
Let $F:T_{II,k}\rightarrow P^kS^1$ be a continuous function. Suppose
$$
F(t)=[x_1(t),x_2(t),\ldots,x_k(t)],
$$
and for all $t\in T_{II,k}$, $x_i(t)\neq x_j(t)$ for $i\neq j$.
Then there are continuous functions $f_1,f_2,\ldots f_k: T_{II,k}\rightarrow S^1$
such that
$$
F(t)=[f_1(t),f_2(t),\ldots, f_k(t)].
$$
\end{lem}

\begin{proof}  Note that the restriction of the map  $\pi: (S^1)^k \to P^kS^1$ on $\pi^{-1}(\mbox{\r{P}}^k S^1)$ is a covering map and $\pi_1(T_{II,k})=\bZ/k\bZ$ is a torsion group. The lemma follows from Lemma \ref{L:pi_1} and the lifting lemma Proposition 1.33 in \cite{Hatcher}.
\end{proof}

\begin{thm}\label{T:T2k}
Let $u\in CU(PM_n(C(T_{II,k}))P)$, then for any $\eps>0$,
there exists a self-adjoint element $h\in PM_n(C(T_{II,k}))P$
with $\|h\|\leq 1$ such that $\|u-exp(2\pi ih)\|<\eps$.
In particular, $cel(u)\leq2\pi$.
\end{thm}

\begin{proof}The proof is inspired by the proof of Lemma 3 of \cite{lin}
(see also Remark 3.11 of \cite{GLN1}).
By Lemma \ref{L:transversal} and Lemma \ref{L:lifting}, for each $\eps>0$,
there exists $v\in CU(PM_n(C(T_{II,k}))P)$ with $\|u-v\|\leq \eps$ and
$Sp(v(y))=\{\beta_1(y),\beta_2(y),\ldots\beta_q(y)\}$, where
$q=rank(P)$ and $\beta_l(y)\neq\beta_j(y)$ for all $l\neq j$ and $y\in T_{II,k}$.
Further, $\beta_j:T_{II,k}\rightarrow S^1$ is continuous for each $1\leq j\leq q$.
Also, $v$ can be chosen in $CU(PM_n(C(T_{II,k}))P)$.

Arbitrarily choose a point $y_0\in T_{II,k}$.
We can choose some real $b_j\in C(T_{II,k})$ such that $\beta_j(y)=exp(2\pi ib_j(y))$, where $b_j(y_0)\in(-\frac{1}{2},\frac{1}{2}]$, $j=1,2,\ldots,q$.
Since $v\in CU(PM_n(C(T_{II,k}))P)$, we have $det(v(y))=1$ for each $y\in T_{II,k}$. Then $\sum_{j=1}^qb_j(y_0)=m$ for some integer $m$.
Since $b_j(y_0)\in(0,1]$, we have $-q<m<q$.

If $m\geq 1$, without loss of generality, we can assume that $b_1(y_0)>b_2(y_0)>\cdots>b_q(y_0)$. It follows that $b_m(y_0)>0$.
Define $a_j(y)=b_j(y)-1$, $j=1,2\ldots,m$, $y\in T_{II,k}$ and $a_j(y)=b_j(y)$ for $j>m$, $y\in T_{II,k}$.

Then
$$
\sum_{j=1}^qa_j(y_0)=0,~~\text{and}~~ |a_j(y_0)|<1.\eqno(4.1)
$$

Also, since $b_j(y_0)>-\frac{1}{2}$, we have $\min_j a_j(y_0)=b_m(y_0)-1$.
Note that $\max_j a_j(y_0)<b_m(y_0)$, we have
$$
\max_j a_j(y_0)-\min_j a_j(y_0)<1. \eqno(4.2)
$$

If $m\leq -1$, we directly assume that $b_1(y_0)<b_2(y_0)<\cdots<b_q(y_0)$.
It follows that $b_m(y_0)<0$. Define $a_j(y)=b_j(y)+1$ for $j=1,2,\ldots, m$,
$y\in T_{II,k}$ and $a_j(y)=b_j(y)$ for $j>m$, $y\in T_{II,k}$.
Then (4.1) and (4.2) also hold.

%Since $\beta_l(y)\neq\beta_j(y)$ for all $l\neq j$ and $y\in T_{II,k}$,
%we have $e^{2\pi i b_l(y)}\neq e^{2\pi i b_j(y)}$. It follows that
%$b_l(y)-b_l(y)\notin \mathbb{Z}$. Since each $b_j$ is continuous and
%$T_{II,k}$ is connected, we have
%$$
%\max_{1\leq j\leq q}a_j(y)-\min_{1\leq j\leq q}a_j(y)<1~~\text{for all}~~y\in T_{II,k}.\eqno(4.3)
%$$

Hence $\beta_j(t)=exp(2\pi ib_j(y))=exp(2\pi ia_j(y))$ for each $1\leq j\leq q$.
Since $det(v(y))=1$ for all $y\in T_{II,k}$, we have
$$
\sum_{j=1}^qa_j(y_0)\in \mathbb{Z}~~\text{for all}~~y\in T_{II,k}.
$$
Since $\sum_{j=1}^qa_j\in C(T_{II,k})$ and $T_{II,k}$ is connected,
it follows that it is a constant.
By (4.1), we have
$$
\sum_{j=1}^qa_j(y)=0 ~~\text{for all}~~ y\in T_{II,k}.\eqno(4.3)
$$
Since $\beta_l(y)\neq \beta_j(y)$ for any $l\neq j$
and $y\in T_{II,k}$, we have
$$
a_l(y)-a_j(y)\notin\mathbb{Z}~~\text{ for all}~~y\in T_{II,k}, l\neq j.
$$
Note that $\max_{1\leq j\leq q}a_j(y)-\min_{1\leq j\leq q}a_j(y)$
is a continuous function on $T_{II,k}$ and $T_{II,k}$ is connected,
by (4.2), we have
$$
0<\max_j a_j(y)-\min_j a_j(y)<1~~\text{for all}~~y\in T_{II,k}.\eqno(4.4)
$$

By (4.3), either $a_j(y)=0$ for all $1\leq j\leq q$, which is impossible,
since $a_j(y)\neq a_l(y)$ when $j\neq l$, or, $h_j(y)<0$ for some $j$
and $h_l(y)>0$ for other $l$. By (4.4), we have
$$
|a_j(y)|<1 ~~\text{for all}~~y\in T_{II,k}.
$$

Fixed $1\leq j\leq q$. For any $y\in  T_{II,k}$, let $p_j(y)$ be the spectrum projection of $v(y)$ with respect to the spectrum $exp(2\pi i a_j(y))$,
which is  rank one projection continuously depending on $y$. Then
$v(y)=\sum_{j=1}^q exp(2\pi i a_j(y)) p_j(y)$.

We let $h\in (PM_n(C(T_{II,k}))P)_{s.a}$ be defined by $h(y)=\sum_{j=1}^q a_j(y) p_j(y)$. Then $\|h\|\leq 1$ and $v=exp(2\pi ih)$, Consequently, $\|u-exp(2\pi ih)\|<\eps$.

%$$
%v_t(y)=P(y)V_y
%\begin{bmatrix}
%e^{2\pi i ta_1(y)}&&&&&\\
%&\ddots&&&&\\
%&&e^{2\pi i ta_q(y)}&&&\\
%&&&0&&\\
%&&&&\ddots&\\
%&&&&&0
%\end{bmatrix}V_y^*P(y),~~\text{for all}~~t\in[0,1].
%$$
%It is easy to check that
%$$
%length(v_t)=\int_{0}^1\|\frac{dv}{dt}\|dt\leq\max_{1\leq j\leq q}\max_{y\in T_{II,k}}2\pi|a_j(y)|\leq2\pi.
%$$
\end{proof}

Using a similar method, we  can get the following result.

\begin{thm}\label{T:thmblock}
Let $u\in CU(PM_n(C(X))P)$, where $X$ is one of the space $\{pt\}$,
$[0,1]$ and $S^1$ and $P$ is a projection in $M_n(C(X))$. Then for
any $\eps>0$, there exists a self-adjoint element $h\in PM_n(C(X))P$
with $\|h\|\leq1$ such that $\|u-exp(2\pi ih)\|<\eps$.
In particular, $cel(u)\leq 2\pi$.
\end{thm}

Now we are going to prove  the following result.
Its proof is similar to Lemma 3.12 in \cite{GLN1}.

\begin{thm}\label{T:dimensionatrop}
Let $u\in CU(M_l(I_k))$. Then for
any $\eps>0$, there exists a self-adjoint element $h\in M_l(I_k)$
with $\|h\|\leq1$ such that $\|u-exp(2\pi ih)\|<\eps$.
In particular, $cel(u)\leq 2\pi$.
\end{thm}

\begin{proof}
For any $u\in CU(M_l(I_k))$ and $\eps>0$,
there exists a unitary $v\in M_l(I_k)$ such that
$$
\|u-v\|<\eps,
$$
where $v$ can be written as $v(t)=U(t)^*exp(2\pi iH(t))U(t)$,
$H(t)$ is a self-adjoint element in $M_l(I_k)$ with $\|H(t)\|\leq 2\pi$.
Then $cel(v)\leq 2\pi$.
\end{proof}

Now we get the following result.

\begin{thm}\label{T:idealproperty}
Let $A$ be an $AH$ algebra with ideal property of no dimension growth.
Then  for any $\eps>0$ and any $u\in CU(A)$, there exists a self-adjoint element $h$ in $A$ with $\|h\|\leq1$ such that $\|u-exp(2\pi i h)\|<\eps$. In particular,
$cel_{CU}(A)\leq 2\pi$
\end{thm}

\begin{proof}
We assume that $A=\lim(A_n, \phi_{n,n+1})$, where $A_n\in \mathcal{HD}$ for each $n\geq 1$.
Using Theorem \ref{T:thmblock} and \ref{T:dimensionatrop}, for each $u\in CU(A_n)$
we have $cel(\phi_{n,m}(u))\leq 2\pi$ for each $m\geq n$. Noting that $cel(\phi_{n,\infty}(u))=\inf_{m\geq n}cel(\phi_{n,m}(u))\leq 2\pi$,
hence $cel(\phi_{n,\infty}(u))\leq 2\pi.$
\end{proof}

The above theorem generalize Theorem 4.6 of \cite{lin} (see Theorem A in  the introduction) for the case of simple $AH$ algebra.

\vspace{0.2in}

%\begin{rem}\label{R:ideal}
%Use a same method as the proof in Proposition \ref{L:variation},
%one can prove that the result of Proposition \ref{L:variation} also holds for $AH$
%inductive limit with no dimension growth.
%\end{rem}

The following Theorem is the main theorem of this section. This theorem is not quite a consequence of Theorem \ref{T:Main} and Theorem \ref{T:idealproperty}, since it is not assume that $\lim_{n\to \infty} rank(P_{n,i})=\infty$. But
we assume $A$ has ideal property.

\begin{thm}\label{T:idealproperty1}
Let $A$ be an $AH$ algebra with ideal property of no dimension growth. If we further assume that $A$ is not of real rank zero,
then $cel_{CU}(A)=2\pi$.
\end{thm}

%{\blue{give a proof by use Pasnicu's dichotomy.}}

To prove the above result, we need the following Pasinicu's dichotomy lemma.

\begin{prop}[\cite{Pasnicu}, Lemma 2.11]\label{P:dichotomy}
Let $A=\lim_{\rightarrow}(A_n=\oplus_{i=1}^{t_n}P_{n,i}M_{[n,i]}(C(X_{n,i}))P_{n,i},\phi_{n,m})$
be an $AH$ algebra with the ideal property and with no dimension growth condition. Then  for any $n$, any finite subset $F_n^i\subset P_{n,i}M_{[n,i]}(C(X_{n,i}))P_{n,i}\subset A_n$, any $\eps>0$ and any
positive integer $N$, there exists $m_0>n$ such that each partial map
$\phi_{n,m}^{i,j}$ with $m\geq m_0$ satisfies either:
\begin{enumerate}
\item[(1)] $rank(\phi_{n,m}^{i,j}(P_{n,i}))\geq N(\dim X_{m,j}+1)$, or
\item[(2)] there is a homomorphism
$$
\psi_{n,m}^{i,j}: A_n^i\rightarrow \phi_{n,m}^{i,j}(P_{n,i})A_m^j\phi_{n,m}^{i,j}(P_{n,i})
$$
with finite dimensional image such that
$\|\phi_{n,m}^{i,j}(f)-\psi_{n,m}^{i,j}(f)\|<\eps$ for all $f\in F_n^i$.
\end{enumerate}
\end{prop}

\begin{rem}\label{P:rem}  Let $X,Y$ be connected finite simplicial complxes. Let $f\in PM_n(C(X))P$ be a self-adjoint element and $\phi, \psi:PM_n(C(X))P \to QM_m(C(Y))Q$ be two unital homomorphisms with $\psi$  factoring through a finite dimensional algebra such that  $\|\phi(f)-\psi(f)\|< \eps$. Then all functions in the eigenvalue list of $\psi(f)$ are constant functions and consequently $EV(\psi(f))=0$. Also by Weyl inequality \cite{Weyl}, $EV(\phi(f))< \eps$.

\end{rem}

\begin{proof}[\emph{\bf Proof of Theorem \ref{T:idealproperty1}}]

Since $A$ is not real rank zero, by Proposition \ref{L:variation} (for the case of no dimension growth),
there exist $\delta_0>0$, $N$, $x\in (A_{N})_+$
with $\|x\|=1$ and a subsequence $\{A_{n_k}\}_{k=2}^\infty$
with $n_2>N$ such that for each $k\geq 2$,
there is a block $A_{n_k}^j$ with
$$
EV(\phi_{N,n_k}^{-,j}(x))\geq \delta_0{\red{.}}\eqno(4.5)
$$
To save notations, we can directly assume that $N=1$
and $n_k=k$ for every $k\geq 2$.

%We denote
%$\Gamma=\{1\leq i\leq t_1: P_{1,i}x\neq 0\}$. Then
%$$
%x=\oplus_{i\in \Gamma} x_i,
%$$
%where $x_i=P_{1,i}x$.

For any $\eps>0$ and $N>\frac{4}{\eps}$, by Proposition \ref{P:dichotomy}
and Remark \ref{P:rem}, there exists $m_0>1$ such that
for each block $A_m^j=P_{m,j}M_{[m,j]}(C(X_{m,j}))P_{m,j}$, either:
$$
rank(P_{m,j})\geq N(\dim X_{m,j}+1),
$$
or
$$
EV(\phi_{1,m}^{-,j}(x))<\delta_0{\red{.}}
$$

We denote
$$
\Lambda=\{1\leq j\leq t_{m_0}: rank(P_{m,j})\geq N(\dim X_{m_0,j}+1)~
\}.
$$
Let $P=\oplus_{j\in\Lambda}P_{m_0,j}$ and $R=\oplus_{j\notin \Lambda}P_{m_0,j}$ and $x^j=\phi_{1, m_0}^{-,j}(x)$. Set $x_1=\oplus_{j\in \Lambda} x^j =P\phi_{1,m_0}(x)P $ and $x_2=\oplus_{j\notin \Lambda} x^j =R\phi_{1,m_0}(x)R$. From above, we know that $EV(x^j)<\delta_0$ for $j\notin \Lambda$. By Corollary \ref{june-17}, for any $m>m_0$ and any $ j\in\{1,2,\cdots, t_m\} $, $\phi_{m_0, m}^{-, j}(x_2)< \delta_0$. By (4.5), for each $m>m_0$, there is a $j\in \{1,2,\cdots,t_m\}$ such that $\phi_{1, m}^{-,j}(x)\geq \delta_0$. Note that $\phi_{1,m}^{-,j}(x)=\phi_{m_0, m}^{-,j}(x_1)+\phi_{m_0,m}^{-,j}(x_2)$. By \ref{june-17} (b), $EV(\phi_{m_0, m}^{-,j}(x_1))\geq\delta_0$.

Hence by Proposition \ref{L:variation},
$\phi_{m_0,\infty}(P)A\phi_{m_0,\infty}(P)=\lim(\phi_{m_0, m}(P) A_m\phi_{m_0, m}(P), \phi_{m, m'})$ is not real rank zero{\red{.}}

%We denote
%$$
%x=\{x_j\}_{j=1}^{k_{m_0}}\in A_{m_0}=\oplus_{k=1}^{k_{m_0}}P_{m_0,k}M_{[m_0,k](C(X_{m_0,k}))}P_{m_0,k}.
%$$
By Theorem 1.2 in Page 112 of \cite{Husemoller}, for each $j\in\Lambda$,
there exists a set of mutually orthogonal rank one projections
$p_1^{(j)}, p_2^{(j)},\cdots, p_L^{(j)}$ with $p_l^{(j)}<P_{m_0,j}$
and $p_l^{(j)}\sim p_1^{(j)}$ for each $1\leq l\leq L$.
We let $q=\oplus_{j\in \Lambda}p^{(j)}_1$.
There exists some $W\in \mathbb{N}$ such that
$$
P<W[q].
$$
It follows that $\phi_{m_0,\infty}(q)A\phi_{m_0,\infty}(q)$
is stably isomorphic to $\phi_{m_0,\infty}(P)A\phi_{m_0,\infty}(P)$,
and hence  $\phi_{m_0,\infty}(q)A\phi_{m_0,\infty}(q) =\lim(\phi_{m_0,m}(q)A_{m}\phi_{m_0,m}(q), \phi_{m,m'})$
is not of real rank zero.
By Lemma \ref{R:realrank1}, there exist an interval $[c,d]\subset[0,1]$,
an integer $m_1\geq m_0$, and $y\in (\phi_{m_0,m_1}(q)A_{m_1}\phi_{m_0,m_1}(q))_+$
with $\|y\|=1$ such that for each $m\geq m_1$, $\widetilde{\phi}_{m_1,m}(y)$
has the following representation
$$
\widetilde{\phi}_{m_1,m}(y)=(z_1^{m},z_2^{m},\cdots,z_{k_m}^{m})\in \bigoplus_{j=1}^{k_m}\phi_{m_0,m}^{-,j}(q) A_m^j \phi_{m_0,m}^{-,j}(q)=\phi_{m_0, m}(q)A_m\phi_{m_0,m}(q),
$$
where $\widetilde{\phi}_{m_1,m}=\phi_{m_1,m}|_{\phi_{m_0,m_1}(q)A_{m_1}\phi_{m_0,m_1}(q)}$,
there exists $1\leq k(m)\leq k_m$, $1\leq i(m)\leq [m,k(m)]$ such that
$$
[c,d]\subset rang(h^{k(m)}_{i(m)}),
$$
where $h_i^{k(m)}$ is the $i$-th lowest eigenvalue of $y_{k(m)}^m$ for $1\leq i\leq [m,k(m)]$.

We let $Q=\sum_{l=1}^L\oplus_{j\in \Lambda}p_l^{(j)}$, then $QA_{m_0}Q$ and $M_L(qA_{m_0}q)$ are isomorphic. This means that $M_L(qA_{m_0}q)\subset A_{m_0}$
and hence $M_L(\phi_{m_0,m_1}(q)A_{m_1}\phi_{m_0,m_1}(q))\subset A_{m_1}$.

Applying a similar proof as in Theorem \ref{T:Main}, we can prove that
$$
cel_{CU}(A)\geq 2\pi.
$$

\end{proof}

%%%%%%%%%%%%%%%%%%%%%%%%%%%%%%%%%%%%%%%%%%%%%%%%%%%%%%
\section{Exponential length in the Jiang-Su algebra}
%%%%%%%%%%%%%%%%%%%%%%%%%%%%%%%%%%%%%%%%%%%%%%%%%%%%%%

We shall show that there exists $u\in CU(\mathcal{Z})$
such that $cel(u)\geq 2\pi$.
First, we review the construction of the Jiang-Su algebra
$\mathcal{Z}$. We refer the readers  to \cite{JiangSu} for details.
We denote by $I[m_0,m,m_1]$ the following dimension drop algebra:
$$
\{f\in C([0,1],M_m):f(0)\in M_{m_0}\otimes {\bf 1}_{m/{m_0}},f(1)\in {\bf 1}_{m/{m_1}}\otimes M_{m_1}\},
$$
where $m_0$, $m_1$ and $m$ are positive integers with $m$ divisible by both
$m_0$ and $m_1$. If $m_0$ and $m_1$ are relatively prime, and $m=m_0m_1$,
then $I[m_0,m,m_1]$ is called a prime dimension drop algebra.

The Jiang-Su algebra is constructed as below.
We let $A_1=I[2,6,3]$. Suppose that a prime dimension drop
algebra $A_m=I[p_m,d_m,q_m]$ is chosen for some $m\geq 1$.
We construct $A_{m+1}$ and $\phi_{m,m+1}:A_m\rightarrow A_{m+1}$ as follows.

Choose $k_0^{(m)}$ and $k_1^{(m)}$ to be the first two prime numbers that are
greater than $2d_m$. This means that
$$
k_0^{(m)}>2p_m, ~~k_1^{(m)}>2q_m, ~~(k_0^{(m)}p_m,k_1^{(m)}q_m)=1.
$$
Let $p_{m+1}=k_0^{(m)}p_m$, $q_{m+1}=k_1^{(m)}q_m$, $d_{m+1}=p_{m+1}q_{m+1}$ and
$A_{m+1}=I[p_{m+1},d_{m+1},q_{m+1}]$. Obviously, $A_{m+1}$ is a prime
dimension drop algebra. Denote $k^{(m)}=k_0^{(m)}k_1^{(m)}$. Choose
$r_0^{(m)}$ such that
$$
0<r_0^{(m)}\leq q_{m+1}, ~~q_{m+1}|(k^{(m)}-r_0^{(m)}).
$$
Choose $r_1^{(m)}$ such that
$$
0<r_1^{(m)}\leq p_{m+1}, ~~p_{m+1}|(k^{(m)}-r_1^{(m)}).
$$
Define
\begin{equation*}
\xi_j^{(m)}(x)=
\begin{cases}
x/2,\qquad\quad 1\leq j\leq r_0^{(m)}\\
1/2,\qquad\quad r_0^{(m)}<j\leq k^{(m)}-r_1^{(m)}\\
(x+1)/2,\quad k^{(m)}-r_1^{(m)}<j\leq k^{(m)}
\end{cases}.
\end{equation*}
It follows that
\begin{equation*}
\xi_j^{(m)}(0)=
\begin{cases}
0,\qquad 1\leq j\leq r_0^{(m)}\\
1/2,\qquad r_0^{(m)}<j\leq k^{(m)}
\end{cases}, and
\end{equation*}

\begin{equation*}
\xi_j^{(m)}(1)=
\begin{cases}
1/2,\qquad 1\leq j\leq k^{(m)}-r_1^{(m)}\\
1,\qquad k^{(m)}-r_1^{(m)}<j\leq k^{(m)}
\end{cases}.
\end{equation*}

Obviously, we have
$$
r_0^{(m)}q_m\equiv k^{(m)}q_m=k_0^{(m)}q_{m+1}\equiv0,~~~(mod ~~~q_{m+1}).
$$
It follows that $q_{m+1}|r_0^{(m)}q_m$.
Notice that $q_{m+1}|(k^{(m)}-r_0^{(m)})$, there exists a unitary
$u_0\in M_{d_{m+1}}$ such that
$$
\rho_0(f)=u_0^*
\begin{bmatrix}
f(\xi_1^{(m)}(0))&0&\cdots&0\\
0&f(\xi_2^{(m)}(0))&\cdots&0\\
\vdots&\vdots&&\vdots\\
0&0&\cdots&f(\xi_{k^{(m)}}^{(m)}(0))	
\end{bmatrix}u_0,~~~\text{for all}~~ f\in A_m,
$$
defines a morphism $\rho_0:A_m\rightarrow M_{p_{m+1}}\otimes {\bf 1}_{q_{m+1}}\subset M_{p_{m+1}}\otimes M_{q_{m+1}}$.

On the other hand, we have
$$
p_{m+1}|r_1^{(m)}p_m,~~~p_{m+1}|(k^{(m)}-r_1^{(m)}).
$$
There exists a unitary $u_1\in M_{d_{m+1}}$ such that
$$
\rho_1(f)=u_1^*
\begin{bmatrix}
f(\xi_1^{(1)}(1))&0&\cdots&0\\
0&f(\xi_2^{(2)}(1))&\cdots&0\\
\vdots&\vdots&&\vdots\\
0&0&\cdots&f(\xi_{k^{(m)}}^{(m)}(1))	
\end{bmatrix}u_1,~~~\text{for all}~~ f\in A_m,
$$
defines a morphism $\rho_1: A_m\rightarrow {\bf 1}_{p_{m+1}}\otimes M_{q_{m+1}}\subset M_{p_{m+1}}\otimes M_{q_{m+1}}$.

Let $u$ be any continuous path of unitaries in $M_{d_{m+1}}$
connecting $u_0$ and $u_1$ and let $\phi_{m,m+1}$ be given
as follows
$$
\phi_{m,m+1}(f)=u^*
\begin{bmatrix}
f\circ\xi_1^{(m)}&0&\cdots&0\\
0&f\circ\xi_2^{(m)}&\cdots&0\\
\vdots&\vdots&&\vdots\\
0&0&\cdots&f\circ\xi^{(m)}_{k^{(m)}}	
\end{bmatrix}u,~~~\text{for all}~~ f\in A_m.
$$

\begin{thm}[\cite{JiangSu}, Proposition 2.5]\label{T:Z}
Jiang-Su algebra $\mathcal{Z}$ can be written  as the
limit $\mathcal{Z}=\lim_{n}(A_n,\phi_{n,n+1})$,
such that each connecting map $\phi_{m,n}=\phi_{n}\circ\phi_{n-1}\circ\cdots\circ\phi_{m+1}\circ\phi_m$
has the following form:
$$
\phi_{m,n}(f)=U^*
\begin{bmatrix}
f\circ\xi_1&0&\cdots&0\\
0&f\circ\xi_2&\cdots&0\\
\vdots&\vdots&\ddots&\vdots\\
0&0&\cdots&f\circ\xi_k
\end{bmatrix}U,
$$
where $U$ is a continuous path in $U(M_{d_n})$, $k=k^{(m)}k^{(m+1)}\cdots k^{(n-1)}$
and
$$
\xi_1\leq\xi_2\leq\cdots\leq \xi_k,
$$
In fact, each $\xi_j$ can be chosen from the following list:
$$
\xi(t)=\frac{l}{2^{n-m}}, ~~\text{where}~~l\in\mathbb{Z},~0<l<2^{n-m},
$$
or
$$
\xi(t)=\frac{t+l}{2^{n-m}}, ~~\text{where}~~l\in\mathbb{Z},~0<l<2^{n-m}.
$$
\end{thm}

\begin{rem}\label{R:JSalgebra}
In the remark we will use the concept of sets with multiplicity.
We shall use $x^{\sim k}$ for a simplified notation for $\{x,x,\ldots,x\}$.
For example $\{x^{\sim2},y^{\sim3}\}=\{x,x,y,y,y\}$. As the construction of
$\mathcal{Z}$, we have
$$
\{\xi_1(0),\xi_2(0),\cdots,\xi_k(0)\}=\{{\frac{l}{2^{n-m}}}^{\sim j_l}\}_{l=0}^{2^{n-m}-1}
$$
where $q_n|j_l$ for all $0\leq l\leq 2^{n-m}-1$, and
$$
\{\xi_1(1),\xi_2(1),\cdots,\xi_k(1)\}=\{{\frac{l}{2^{n-m}}}^{\sim s_l}\}_{l=1}^{2^{n-m}}
$$
where $p_n|s_l$ for all $1\leq l\leq 2^{n-m}$.
\end{rem}

\begin{lem}\label{L:per}
Let $\mathcal{Z}=\lim(A_n,\phi_{n,n+1})$ be the Jiang-Su algebra, which
is defined as above. If $v\in A_n$ is a unitary and $u_s$ is a smooth path of unitaries
connecting $v$ and ${\bf 1_{A_n}}$, then for any $\eps>0$,
there exists another smooth path $v_s\in A_n$ of unitaries such that
\begin{enumerate}
\item $\|v_s-u_s\|<\eps$,
\item $|length(v_s)-length(u_s)|<\eps$,
\item $v_s(0)=exp(2\pi i \sum_{j=1}^{p_n}h_ja_j)\otimes {\bf 1}_{q_n}$, 
$\{a_j\}_{j=1}^{p_n}$ is a set of mutually orthogonal
    rank one projections in $C([0,1],M_{p_n})$ and
    $h_j\in C([0,1])_{s.a.}$, $exp(2\pi i h_j(s))\neq exp(2\pi i h_k(s))$ for each $j\neq k$ and $s\in[0,1]$.
\item $v_s(1)=exp(2\pi i \sum_{j=1}^{q_n}g_jb_j)\otimes {\bf 1}_{p_n}$, 
$\{b_j\}_{j=1}^{q_n}$ is a set of mutually orthogonal
    rank one projections in $C([0,1],M_{q_n})$ and
    $g_j\in C([0,1])_{s.a.}$, $exp(2\pi i g_j(s))\neq exp(2\pi i g_k(s))$ for each $j\neq k$ and $s\in[0,1]$.

\end{enumerate}	
\end{lem}

\begin{proof}
For any $0<\eps<1$, we can find
$\delta>0$ such that $\|u_s(t_1)-u_s(t_2)\|<\frac{\eps}{2}$ and
$\|\frac{du_s}{ds}(t_1)-\frac{du_s}{ds}(t_2)\|<\frac{\eps}{3}$ for any $s\in[0,1]$
and $t_1,t_2\in[0,1]$ with $|s_1-s_2|<\delta$. 

Since $u_s\in A_n$, then we can write $u_s(0)$ as 
$u_s(0)=\gamma^{(0)}(s)\otimes {\bf 1}_{q_n}$, where
$\gamma^{(0)}$ is a unitary in $C([0,1], M_{p_n})$. By Lemma 4.1 in \cite{lin},
we can find a set of mutually orthogonal rank one projections $\{a_j\}_{j=1}^{p_n}$ in $C([0,1], M_{p_n})$ and $h_j\in C([0,1])_{s.a.}$ with $exp(2\pi i h_j(s))\neq exp(2\pi i h_k(s))$ for each $j\neq k$ and $s\in[0,1]$ such that
$$
\|\gamma^{(0)}(s)-{\overline{\gamma}}^{(0)}(s)\|<\frac{\eps}{6(1+\max_{s\in[0,1]}\|\frac{du_s(0)}{ds}\|)}, ~~\|\frac{d \gamma^{(0)}(s)}{d s}-\frac{d {\overline{\gamma}}^{(0)}(s)}{ds}\|<\frac{\eps}{3},
~~\text{ for all}~s\in[0,1],
$$
where ${\overline{\gamma}}^{(0)}(s)=exp(2\pi i\sum_{j=1}^{p_n}h_ja_j)$.

On the other hand, we can write $u_s(1)$ as 
$u_s(1)=\gamma^{(1)}(s)\otimes {\bf 1}_{p_n}$, where
$\gamma^{(1)}$ is a unitary in $C([0,1], M_{q_n})$. By Lemma 4.1 in \cite{lin},
we can find a set of mutually orthogonal rank one projections $\{b_j\}_{j=1}^{q_n}$ in $C([0,1], M_{q_n})$ and $g_j\in C([0,1])_{s.a.}$ with $exp(2\pi i g_j(s))\neq exp(2\pi i g_k(s))$ for each $j\neq k$ and $s\in[0,1]$ such that
$$
\|\gamma^{(1)}(s)-{\overline{\gamma}}^{(1)}(s)\|<\frac{\eps}{6(1+\max_{s\in[0,1]}\|\frac{du_s(1)}{ds}\|)},~~ \|\frac{d \gamma^{(1)}(s)}{d s}-\frac{d {\overline{\gamma}}^{(1)}(s)}{ds}\|<\frac{\eps}{3},
~~\text{ for all}~s\in[0,1],
$$
where ${\overline{\gamma}}^{(1)}(s)=exp(2\pi i\sum_{j=1}^{q_n}g_jb_j)$.

We denote $v^{(0)}(s)={\overline{\gamma}}^{(0)}(s)\otimes {\bf 1}_{q_n}$
and $v^{(1)}(s)={\overline{\gamma}}^{(1)}(s)\otimes {\bf 1}_{p_n}$.
Then 
$$
\|v^{(0)}(s)-u_s(0)\|=
\|(\gamma^{(0)}(s)-{\overline{\gamma}}^{(0)}(s))\otimes {\bf 1}_{q_n}\|<\frac{\eps}{6(1+\max_{s\in[0,1]}\|\frac{du_s(0)}{ds}\|)},
$$
and
$$
\|v^{(1)}(s)-u_s(1)\|=
\|(\gamma^{(1)}(s)-{\overline{\gamma}}^{(1)}(s))\otimes {\bf 1}_{p_n}\|<\frac{\eps}{6(1+\max_{s\in[0,1]}\|\frac{du_s(1)}{ds}\|)},
$$
for all $s\in[0,1]$. Also, we have
$$
\|\frac{dv^{(0)}(s)}{ds}-\frac{du_s(0)}{ds}\|<\frac{\eps}{3},
$$
and
$$
\|\frac{dv^{(1)}(s)}{ds}-\frac{du_s(1)}{ds}\|<\frac{\eps}{3}.
$$

Since $\|u_s^*(0)v^{(0)}(s)-{\bf 1}_{A_n}\|=\|v^{(0)}(s)-u_s(0)\|<\frac{\eps}{6}<\frac{1}{6}$,
there exists $H\in M_{d_n}(C([0,1]))_{s.a.}$ with $\|H\|<1$ such that 
$u_s^*(0)v^{(0)}(s)=exp(2\pi i H(s))$. Also, there exists $G\in M_{d_n}(C([0,1]))_{s.a.}$ with $\|G\|<1$ such that $u_s^*(1)v^{(1)}(s)=exp(2\pi i G(s))$. In fact, $H(s)=\frac{1}{2\pi i}\log(u_s^*(0)v^{(0)}(s))$ and
$G(s)=\frac{1}{2\pi i}\log(u_s^*(1)v^{(1)}(s))$.

We denote
\begin{equation*}
w(s,t)=
\begin{cases}
u_s(0)exp(2\pi i\frac{-t}{\delta}H(s)),\qquad -\delta\leq t<0\\
u_s(t),\qquad\qquad\qquad\qquad\quad 0\leq t\leq 1\\
u_s(1)exp(2\pi i\frac{t-1}{\delta}G(s))\quad \quad1<t\leq 1+\delta
\end{cases}.
\end{equation*}

Let $v_s(t)=w(s,(1+2\delta)t-\delta)$ for $(s,t)\in[0,1]\times[0,1]$.
Then $v_s$ is a path in $A_n$ and it satisfies conditions (3) and (4).

For $t\in[0,\frac{\delta}{1+2\delta})$, by the choice of $\delta$,
we have
\begin{align*}
\|v_s(t)-u_s(t)\|&=\|u_s(0)exp(2\pi i\frac{\delta-(1+2\delta)t}{\delta}H(s))-u_s(t)\|\\
&\leq\|u_s(0)exp(2\pi i\frac{\delta-(1+2\delta)t}{\delta}H(s))-u_s(0)\|+\|u_s(0)-u_s(t)\|\\
&=\|exp(2\pi i\frac{\delta-(1+2\delta)t}{\delta}H(s))-{\bf 1}_{A_n}\|+\|u_s(0)-u_s(t)\|\\
&\leq\|exp(2\pi iH(s))-{\bf 1}_{A_n}\|+\|u_s(0)-u_s(t)\|\\
&< \frac{\eps}{6}+\frac{\eps}{2}\\
&<\eps.
\end{align*}

For $t\in[\frac{\delta}{1+2\delta}, \frac{1+\delta}{1+2\delta}]$, we have
$|(1+2\delta)t-\delta-t|<\delta$ and hence
\begin{align*}
\|v_s(t)-u_s(t)\|=\|u_s((1+2\delta)t-\delta)-u_s(t)\|<\frac{\eps}{2}.
\end{align*}

For $t\in(\frac{1+\delta}{1+2\delta},1]$, by the choice of $\delta$,
we have
\begin{align*}
\|v_s(t)-u_s(t)\|&=\|u_s(1)exp(2\pi i\frac{(1+2\delta)t-\delta-1}{\delta}G(s))-u_s(t)\|\\
&\leq\|u_s(1)exp(2\pi i\frac{(1+2\delta)t-\delta-1}{\delta}G(s))-u_s(1)\|+\|u_s(1)-u_s(t)\|\\
&= \|exp(2\pi i\frac{(1+2\delta)t-\delta-1}{\delta}G(s))-{\bf 1}_{A_n}\|-\|u_s(1)-u_s(t)\|\\
&= \|exp(2\pi iG(s))-{\bf 1}_{A_n}\|-\|u_s(1)-u_s(t)\|\\
&< \frac{\eps}{6}+\frac{\eps}{2}\\
&<\eps.
\end{align*}

It follows that $\|v_s-u_s\|\leq\eps$.

For $t\in [0,\frac{\delta}{1+2\delta})$, a direct calculation shows that

\begin{align*}
|\|\frac{dv_s}{ds}\|-\|\frac{du_s}{ds}\||&=|\sup_{t\in[0,1]}\|u_s(0)exp(2\pi i\frac{\delta-(1+2\delta)t}{\delta}H(s))2\pi i\frac{\delta-(1+2\delta)t}{\delta}\frac{dH(s)}{ds}\\
&+\frac{d u_s(0)}{ds}exp(2\pi i\frac{\delta-(1+2\delta)t}{\delta}H(s))\|-\sup_{t\in[0,1]}\|\frac{du_s(t)}{ds}\||\\
&\leq\sup_{t\in[0,1]}\|u_s(0)exp(2\pi i\frac{\delta-(1+2\delta)t}{\delta}H(s))2\pi i\frac{\delta-(1+2\delta)t}{\delta}\frac{dH(s)}{ds}\|\\
&+\sup_{t\in[0,1]}\|\frac{d u_s(0)}{ds}exp(2\pi i\frac{\delta-(1+2\delta)t}{\delta}H(s))-\frac{du_s(t)}{ds}\|\\
&\leq 2\pi\|\frac{dH(s)}{ds}\|+\sup_{s\in[0,1]}\|\frac{du_s(0)}{ds}-\frac{du_s(t)}{ds}\|+\|\frac{du_s(0)}{ds}\|\sup_{t\in[0,1]}\|1-exp(2\pi i\frac{\delta-(1+2\delta)t}{\delta}H(s))\|\\
& \leq \|-u_s^*(0)\frac{du_s(0)}{ds}u_s^*(0)v^{(0)}(s)+u_s^*(0)\frac{dv^{(0)}(s)}{ds}\|+\frac{\eps}{3}+\|\frac{du_s(0)}{ds}\|\|1-exp(2\pi i H(s))\|\\
&\leq \|\frac{du_s(0)}{ds}-\frac{dv^{(0)}(s)}{ds}\|+2\|\frac{du_s(0)}{ds}\|\|u_s^*(0)v^{(0)}(s)-{\bf 1}_{A_n}\|+\frac{\eps}{3}\\
&<\frac{\eps}{3}+\frac{\eps}{3}+\frac{\eps}{3}\\
&=\eps.
\end{align*}

With a similar calculation, one have
$$
|\|\frac{dv_s}{ds}\|-\|\frac{du_s}{ds}\||<\eps, ~\text{ for all}~t\in[\frac{1+\delta}{1+2\delta},1].
$$

For $t\in[\frac{\delta}{1+2\delta},\frac{1+\delta}{1+2\delta}]$, we have
$|(1+2\delta)t-\delta-t|<\delta$ and hence
\begin{align*}
|\|\frac{dv_s(t)}{ds}\|-\|\frac{du_s(t)}{ds}\||\leq\|\frac{du_s((1+2\delta)t-\delta)}{ds}-\frac{du_s(t)}{ds}\|<\frac{\eps}{3}.
\end{align*}

It follows that $|\|\frac{dv_s}{ds}\|-\|\frac{du_s}{ds}\||<\eps$
for all $s\in[0,1]$. We have
$$
|length_s(v_s)-length_s(u_s)|=|\int_0^1\|\frac{dv_s}{ds}\|ds-\int_0^1\|\frac{du_s}{ds}\|ds|<\eps.
$$
\end{proof}

\begin{rem}\label{R:per}
In above Lemma, we can assume that $h_1(0),h_2(0),\cdots,h_{p_n}(0)\in[0,1]$,
$g_1(0),g_2(0)$, $\cdots,g_{q_n}(0)\in[0,1]$ and
$$
h_1(s)<h_2(s)<\cdots<h_{p_n}(s),
$$	
$$
g_1(s)<g_2(s)<\cdots<g_{q_n}(s)
$$
for all $s\in[0,1]$. Notice that $exp(2\pi i h_j(s))\neq exp(2\pi h_k(s))$,
$exp(2\pi ig_j(s))\neq exp(2\pi ig_k(s))$ for any $j\neq k$ and $s\in[0,1]$,
we have
$$
h_{p_n}(s)-h_1(s)<1, g_{q_n}(s)-g_1(s)<1
$$
for all $s\in[0,1]$.
 \end{rem}

\begin{lem}\label{L:eigenlist}
Let $\mathcal{Z}=\lim(A_n,\phi_{n,n+1})$ be the Jiang-Su algebra, which
is defined as above. If $v\in A_n$ is a unitary and $u_s$ is a path of unitaries
connecting $v$ and ${\bf 1_{A_n}}$, then for any $\eps>0$,
there exists another path $v_s\in A_n$ of unitaries such that
\begin{enumerate}
\item $\|v_s-u_s\|<\eps$,
\item $|length(v_s)-length(u_s)|<\eps$,
\item $v_s(t)=exp(2\pi i H_s(t))$, $H_s(t)=\sum_{j=1}^{d_n}\lambda_s^j(t)p_j(s,t)$,
$\{p_j\}_{j=1}^{d_n}$ is a set of mutually orthogonal
    rank one projections in $C([0,1]\times[0,1],M_{d_n})$ and
    $h_j\in C([0,1]\times[0,1])_{s.a.}$, $exp(2\pi i\lambda_s^j(t))\neq exp(2\pi i\lambda_s^k(t))$ for each $j\neq k$ and $(s,t)\in(0,1)\times[0,1]$.
\end{enumerate}	
\end{lem}

\begin{proof}
For any $0<\eta<1$, since $u_s(t)$ is uniformly continuous on
$[0,1]\times [0,1]$, there exists $\delta_1>0$ such that
$$
\|u_s(t_1)-u_s(t_2)\|<\frac{\eta}{4}
$$
for all $s\in[0,1]$ and $t_1,t_2\in [0,1]$ with $|t_1-t_2|<4\delta_1$.
Since $u_s(0), u_s(1)\in C([0,1])\otimes M_{d_n}$, there are continuous
functions $f_i,g_i:[0,1]\rightarrow \mathbb{C}$, $i=1,2\cdots, d_n$
such that $f_1(s),f_2(s),\cdots,f(s)\}$ are the eigenvalues
for $v_s(0)$ and $\{g_1(s),g_2(s),\cdots,g_{d_n}(s)\}$ are the
eigenvalues for $v_s(1)$, respectively. By Lemma \ref{L:per},
without loss of generality, we may assume that $u_s(0)$
and $u_s(1)$ can be written as the following forms:
$$
u_s(0)=U^{(0)}(s)
\begin{bmatrix}
exp(2\pi if_1(s))&&&\\
&exp(2\pi if_2(s))&&\\
&&\ddots&\\
&&&exp(2\pi if_{d_n}(s))
\end{bmatrix}
(U^{(0)}(s))^*~~\text{for all}~s\in[0,1],
$$
and
$$
u_s(1)=U^{(1)}(s)
\begin{bmatrix}
exp(2\pi ig_1(s))&&&\\
&exp(2\pi ig_2(s))&&\\
&&\ddots&\\
&&&exp(2\pi ig_{d_n}(s))	
\end{bmatrix}
(U^{(1)}(s))^*~~\text{for all}~ s\in[0,1],
$$
where $U^{(0)}$, $U^{(1)}$ are unitaries in $C([0,1])\otimes M_{d_n}$.
By Remark \ref{R:per} , we can assume that $f_1(0),f_2(0),\cdots,f_{d_n}(0)\in[0,1]$,
$g_1(0),g_2(0),\cdots,g_{d_n}(0)\in[0,1]$
and 
$$
f_1(s)\leq f_2(s)\leq\cdots\leq f_{d_n}(s),
$$
$$
g_1(s)\leq g_2(s)\leq \cdots\leq g_{d_n}(s),
$$
for all $s\in[0,1]$. Further, we can assume that
$f_{d_n}(s)-f_1(s)<1$ and $g_{d_n}(s)-g_1(s)<1$ 
for all $s\in[0,1]$.

Let $0<\delta<\delta_1$ be such that
$\frac{2\delta}{1-4\delta}<\delta_1$. We choose 
$h_1,h_2,\cdots,h_{d_n}\in C([0,1])_{s.a.}$
with $h_1(s)<h_2(s)<\cdots<h_{d_n}(s)$ with $h_{d_n}(s)-h_{1}(s)<1$ for
$s\in[0,1]$ such that $\|exp(2\pi ih_j(s))-exp(2\pi if_j(s))\|<\frac{\eta}{4}$
for $1\leq j\leq d_n$ and $s\in[0,1]$. 

Also we can choose
$k_1,k_2,\cdots,k_{d_n}\in C([0,1])_{s.a.}$
with $k_1(s)< k_2(s)<\cdots< k_{d_n}(s)$ for $s\in[0,1]$
such that $\|exp(2\pi ik_j(s))-exp(2\pi ig_j(s))\|<\frac{\eta}{4}$
for $1\leq j\leq d_n$ and $s\in[0,1]$.

We define a new path $\widetilde{u}_s$ as follows:
\begin{equation*}
\widetilde{u}_s(t)=
\begin{cases}
U^{(0)}(s)diag[exp( 2\pi i(\frac{\delta-t}{\delta}f_j(s)+\frac{t}{\delta}h_j(s)))]_{j=1}^{d_n}(U^{(0)}(s))^*,~~~~~~~~~~~~~~~~~~t\in[0,\delta]\\
U^{(0)}(s)diag[exp(2\pi i(\frac{t-\delta}{\delta}f_j(s)+\frac{2\delta-t}{\delta}h_j(s)))]_{j=1}^{d_n}(U^{(0)}(s))^*,~~~~t\in(\delta,2\delta]\\
u_s(\frac{t-2\delta}{1-4\delta}),t\in(2\delta,1-2\delta]\\
U^{(1)}(s)diag[exp(2\pi i(\frac{t-1+2\delta}{\delta}k_j(s)+\frac{1-\delta-t}{\delta}g_j(s)))]_{j=1}^{d_n}(U^{(1)}(s))^*,~~t\in(1-2\delta,1-\delta]
\\
U^{(1)}(s)diag[exp(2\pi i (\frac{t-1+\delta}{\delta}g_j(s)+\frac{1-t}{\delta}k_j(s)))]_{j=1}^{d_n}(U^{(1)}(s))^*,~~t\in(1-\delta,1]
\end{cases}.
\end{equation*}

As in the construction, it is easy to see that $\widetilde{u}_s$ is a path
of unitaries in $A_n$. We have $\widetilde{u}_s(t)$ has no repeat eigenvalues
for $(s,t)\in[0,1]\times(0,\delta]$ or $(s,t)\in[0,1]\times[1-\delta,1)$. Moreover, when $t\in[0,\delta]$,
we have
\begin{align*}
|\widetilde{u}_s(t)-u_s(t)|&\leq |\widetilde{u}_s(t)-\widetilde{u}_s(0)|+
|\widetilde{u}_s(0)-u_s(t)|\\
&\leq\max_{1\leq j\leq d_n}|exp(2\pi i \frac{t}{\delta}(h_j(s)-f_j(s)))-1|+|u_s(0)-u_s(t)|\\
&\leq \max_{1\leq j\leq d_n}|exp(2\pi i(h_j(s)-f_j(s)))-1|+\frac{\eta}{4}\\
&=\max_{1\leq j\leq dn}|exp(2\pi if_j(s))-exp(2\pi ih_j(s))|+\frac{\eta}{4}\\
&\leq\frac{\eta}{2}. 	
\end{align*}
For $t\in(\delta, 2\delta]$, we have
\begin{align*}
|\widetilde{u}_s(t)-u_s(t)|&\leq |\widetilde{u}_s(t)-\widetilde{u}_s(\delta)|+
|\widetilde{u}_s(\delta)-\widetilde{u}_s(2\delta)|+|\widetilde{u}_s(2\delta)-u_s(t)|\\
&\leq \max_{1\leq j\leq d_n}|exp(2\pi i\frac{t-\delta}{\delta
}(f_j(s)-h_j(s)))-1|+\frac{\eta}{4}+|u_s(0)-u_s(t)|\\
&\leq \max_{1\leq j\leq d_n}|exp(2\pi i(f_j(s)-h_j(s)))-1|+\frac{\eta}{4}+\frac{\eta}{4}\\
&=\max_{1\leq j\leq d_n}|exp(2\pi if_j(s))-exp(2\pi ih_j(s))|+\frac{\eta}{2}\\
&\leq\frac{3\eta}{4}.	
\end{align*}
In the same way, we have
$$
|\widetilde{u}_s(t)-u_s(t)|\leq \frac{3\eta}{4}~~\text{for all}~~ t\in(1-2\delta,1].
$$
Further, for $t\in[2\delta,1-2\delta]$, it is easy to see that
$$
|\frac{t-\delta}{1-4\delta}-t|<\frac{2\delta}{1-4\delta}<\delta_1.
$$
Hence
$$
|\widetilde{u}_s(t)-u_s(t)|=|u_s(\frac{t-2\delta}{1-4\delta})-u_s(t)|<
\frac{\eta}{4}.
$$
It follows that
$$
\|\widetilde{u}_s-u_s\|<\frac{3\eta}{4}, \forall s\in[0,1].
$$

In the construction of $\widetilde{u}_s$, it is easy to see that the lengths of $\widetilde{u}_s$ and $u_s$ are close enough if $\eta$ is small enough.

Notice that $\widetilde{u}_0|_{[\delta,1-\delta]}$,
$\widetilde{u}_1|_{[\delta,1-\delta]}$ are unitaries in
$M_{d_n}(C([\delta,1-\delta]))$ with distinct eigenvalues and
$\widetilde{u}_s|_{[\delta,1-\delta]}$ is a path connecting these two
elements. By Proposition \ref{L:pertur} and Remark \ref{R:pertur}, there exists another path
$\widetilde{\widetilde{u}}_s$ in $U(M_{d_n}(C[\delta, 1-\delta]))$ such that
$$
\|\widetilde{\widetilde{u}}_s-\widetilde{u}_s|_{[\delta,1-\delta]}\|<\frac{\eta}{2},
$$
$$
|length(\widetilde{\widetilde{u}}_s)-length(\widetilde{u}_s|_{[\delta,1-\delta]})|<\eta,
$$
$\widetilde{\widetilde{u}}_s$ has no repeat eigenvalues for all
$(s,t)\in[0,1]\times[\delta,1-\delta]$, and
$$
\widetilde{\widetilde{u}}_0(\delta)=\widetilde{u}_s(\delta),
\widetilde{\widetilde{u}}_0(1-\delta)=\widetilde{u}_s(1-\delta).\eqno{(5.1)}
$$
By Proposition \ref{R:remark}, there exsists a unitary $U_s(t)\in C([0,1]\times[\delta,1-\delta],M_{d_n}(\mathbb{C}))$ such that
$$
\widetilde{\widetilde{u}}_s(t)=
U_s(t)\begin{bmatrix}
\xi^1_s(t)&&&\\
&\xi_s^2(t)&&\\
&&\ddots&\\
&&&\xi^{d_n}_s(t)
\end{bmatrix}U_s(t)^*, ~~\text{for all}~~ (s,t)\in[0,1]\times[\delta,1-\delta],
$$
where $\xi_s^j(t):[0,1]\times [\delta,1-\delta]\rightarrow S^1$
is continuous for each $1\leq j\leq d_n$. Since $\mathbb{R}$ is a covering 
space of $S^1$, there exists $\psi_j:[0,1]\times [\delta,1-\delta]\rightarrow \mathbb{R}$ such that 
$$
exp(2\pi i\psi_j(s,t))=\xi_s^j(t), ~~\text{for all}~(s,t)\in[0,1]\times [\delta,1-\delta]. 
$$

By 5.1, without loss of generality, we can assume that
$$
exp(2\pi ih_j(s))=exp(2\pi i\psi_j(s,\delta)),~~\text{for all}~1\leq j\leq d_n,
$$
and
$$
exp(2\pi ik_j(s))=exp(2\pi i\psi_j(s,1-\delta)),~~\text{for all}~1\leq j\leq d_n,
$$
We can choose integers $m_j, l_j$ such that
$$
h_j(s)=\psi_j(s,\delta)+m_j,
$$
and 
$$
\psi_j(s,1-\delta)+m_j=k_j(s)+l_j
$$
for each $1\leq j\leq d_n$ and $s\in [0,1]$.

We denote
\begin{equation*}
\lambda_s^{j}(t)=
\begin{cases}
\frac{\delta-t}{\delta}f_j(s)+\frac{t}{\delta}h_j(s),~~~t\in[0,\delta]\\
\psi_j(s,t)+m_j,~~t\in(\delta,1-\delta]\\
\frac{t-1+\delta}{\delta}g_j(s)+\frac{1-t}{\delta}k_j(s)+l_j,~~t\in(1-\delta,1]	
\end{cases}.
\end{equation*}
Then $\lambda_s^j(t)$ is continuous on $[0,1]\times[0,1]$ and
$exp(2\pi i\lambda_s^j(t))\neq exp(2\pi i\lambda^k_s(t))$ if $j\neq k$ and $(s,t)\in[0,1]\times(0,1)$.

\end{proof}

\begin{thm}\label{T:JiangSu}
Let $\mathcal{Z}$ be the Jiang-Su algebra. Then $cel_{CU}(\mathcal{Z})\geq 2\pi$.	 
\end{thm}

\begin{proof}
Let $\mathcal{Z}=\lim_{m}A_m$ be the Jiang-Su algebra. For each $m\geq1$, we define
a unitary $u\in A_m$ as follows:
$$
u(t)=\begin{bmatrix}
exp(2\pi i h_1(t))&&\\
&\ddots&\\
&&exp(2\pi i h_{d_m}(t))	
\end{bmatrix}_{d_m\times d_m},
$$
where $h_i(t)=\frac{q_m-1}{q_m}t$ for each $1\leq i\leq p_m$,
$h_i(t)=-\frac{1}{q_m}t$ for each $p_m+1\leq i\leq d_m$.
(Here we identify ${\bf 1}_{p_m}\otimes M_{q_m}\ni {\bf 1}\otimes (a_{ij})_{q_m\times q_m}$ with $(a_{ij}{\bf1}_{p_m})\in M_{p_mq_m}$.)
It follows Lemma 3.7 in \cite{GLN1} that $u\in CU(A_m)$. For any fixed
$n\geq m$, denote $v=\phi_{m,n}(u)$.
Let $u_s(t)$ be a unitary path in $A_n$ with $u_0(t)=v(t)$ and $u_1(t)={\bf 1}_{A_n}$.

For any $0<\eps<\frac{1}{2^{n-m}}$, by Lemma \ref{L:eigenlist} , there exists another piecewise smooth unitary path $v_s(t)$ such that
\begin{enumerate}
	\item [(1)] $\|v_s-u_s\|<\frac{\eps}{2}$;
	\item [(2)] $|length_s (v_s)-length_s (u_s)|<\frac{\eps}{2}$;
	\item [(3)] $$
v_s(t)=U_s(t)
\begin{bmatrix}
exp(2\pi if_1(s,t))&&&\\
&exp(2\pi if_2(s,t))&&\\
&&\ddots&\\
&&&exp(2\pi if_{d_n}(s,t))
\end{bmatrix}U_s(t)^*, ~~\text{for all}~~ (s,t)\in[0,1]\times[0,1],
$$
where $f_j(s,t):[0,1]\times[0,1] \rightarrow \mathbb{R}$ is continuous and $exp(2\pi if_j(s,t))\neq exp(2\pi if_k(s,t))$ if $j\neq k$ and $(s,t)\in[0,1]\times(0,1)$.
\end{enumerate}
In the above construction, we can choose $f_j$ such that
$$
\max_{1\leq j\leq d_n}f_j(s,t)-\min_{1\leq j\leq d_n}f_j(s,t)<1,
~~\text{ for all}~~(s,t)\in[0,1]\times(0,1).
$$
In fact, arbitrarily fix a $t_0\in (0,1)$. Without loss of generality,
we can assume that $f_j(0,t_0)\in[0,1)$ for all $1\leq j\leq d_n$.
This means that
$$
\max_{1\leq j\leq d_n}f_j(0,t_0)-\min_{1\leq j\leq d_n}f_j(0,t_0)<1.
$$
Since $exp(2\pi i f_j(s,t))\neq exp(2\pi i f_k(s,t))$ for $j\neq k$, we have $f_j(s,t)-f_k(s,t)\notin \mathbb{Z}$.
Notice that $f_j(\cdot,\cdot)$ is continuous, we have
$$
\max_{1\leq j\leq d_n}f_j(s,t)-\min_{1\leq j\leq d_n}f_j(s,t)<1
~~\text{for all}~~(s,t)\in[0,1]\times(0,1).\eqno(5.2)
$$

Also, we can assume that
$$
f_1(s,t)\leq f_2(s,t)\leq \cdots\leq f_{d_n}(s,t), ~\text{ for all}
~~(s,t)\in[0,1]\times[0,1].\eqno(5.3)
$$

By the construction of the Jiang-Su algebra, we have
$$
v(t)=b^*
\begin{bmatrix}
exp(2\pi i \mu_1(t))&0&\cdots&0\\
0&exp(2\pi i \mu_2(t))&\cdots&0\\
\vdots&\vdots&&\vdots\\
0&0&\cdots&exp(2\pi i \mu_{d_n}(t))	
\end{bmatrix}b,
$$
where $\mu_1(t)\leq \mu_2(t)\leq \cdots\leq \mu_{d_n}(t)$ for all $t\in[0,1]$,
$b$ is a unitary element in $M_{d_n}(C([0,1]))$. There is a permutation
$\sigma$ such that $exp(2\pi if_j(s,\cdot))$ is a path connecting an element near $e^{2\pi i\mu_{\sigma(j)}(t)}$
with some function near $1$ for each $1\leq j\leq d_n$.
Without loss of generality, we assume that $\sigma(j)=j$ for each $j$.

We have
$$
-\frac{1}{q_m}\leq\mu_j(t)\leq \frac{1}{q_m},~~\text{for all}~~ 1\leq j\leq (d_m-p_m)k^{(m)}k^{(m+1)}\dots k^{(n-1)}+p_mr_1^{(m)}r_1^{(m+1)}\dots r_1^{(n-1)},
$$
and
$$
\mu_j(t)=\frac{(q_m-1)(t+2^{n-m}-1)}{q_m2^{n-m}},~~\text{for all}~~ d_n-p_mr_0^{(m)}r_0^{(m+1)}\dots r_0^{(n-1)}<j\leq d_n.
$$
Obviously, we have $length_s(v_s(\cdot))\geq length_s(exp(2\pi if_j(s,\cdot)))$ for each $j$.
It is easy to see that $f_j(s,\cdot)$ is a path connecting an element near $\mu_j(t)$ with an element near $l_j$,
where $l_j\in \mathbb{Z}$, it follows that
$length_s(f_j(s,\cdot))\geq2\pi\max_{t\in[0,1]}|\mu_j(t)-l_j|-\eps$.
We consider the following cases:

{\bf Case 1:} $l_{j_0}\geq 2$ for some $1\leq j_0\leq d_n$.
Then
$$
length_s(e^{2\pi if_{j_0}(s,\cdot)})\geq 2\pi\max_{t\in[0,1]}|\mu_{j_0}(t)-l_{j_0}|-\eps\geq 2\pi-\eps.
$$

{\bf Case 2:} $l_{j_1}\leq -1$ for some $1\leq j_1\leq d_n$.
Then
$$
length_s(e^{2\pi if_{j_1}(s,\cdot)})\geq 2\pi\max_{t\in[0,1]}|\mu_{j_1}(t)-l_{j_1}|-\eps\geq 2\pi-\eps.
$$

{\bf Case 3:} $l_j=0$ for all $1\leq j\leq d_n$.
Then
$$
length_s(e^{2\pi if_{d_n}(s,\cdot)})\geq 2\pi\max_{t\in[0,1]}|\mu_{d_n}(t)-0|-\eps=2\pi\frac{q_m-1}{q_m}-\eps.
$$

{\bf Case 4:} $l_j=1$ for all $1\leq j\leq d_n$.

$$
length_s(e^{2\pi if_{1}(s,\cdot)})\geq 2\pi\max_{t\in[0,1]}|\mu_{1}(t)-1|-\eps=2\pi\frac{q_m+1}{q_m}-\eps.
$$

{\bf Case 5:} All $l_j$ are either $0$ or $1$,
it follows that $\{j:l_j=0\}\neq \emptyset$ and
$\{j:l_j=1\}\neq \emptyset$.
%Without loss of generality, we assume that
%$$
%l_j=0~\text { for all}~ 1\leq j\leq (d_m-p_m)k^{(m)}k^{(m+1)}\cdots k^{(n-1)}+p_mr_1^{(m)}r_1^{(m+1)}\cdots r_1^{(n-1)}.\eqno(5.4)
%$$
%Otherwise, as the Case 4, we can prove that
%$$
%length_s(e^{2\pi if_j(s,\cdot)})\geq 2\pi\frac{q_m-1}{q_m}-\eps,
%$$
%for some $j$. For the same reason, we can assume that
%$$
%l_j=1~\text{ for all}~ d_n-p_mr_0^{(m)}r_0^{(m+1)}\cdots r_0^{(n-1)}<j\leq d_n.\eqno(5.5)
%$$
By (5.3), we can assume that there exists $1\leq K<d_n$
such that $l_j=0$ for all $1\leq j\leq K$ and $l_j=1$ for all $K+1\leq j\leq d_n$.\\

{\bf Claim 1.} If $q_n\nmid K$, then $length_s(v_s)\geq 2\pi
(\frac{(q_m-1)(2^{n-m}-1)}{q_m2^{n-m}}-\eps)$.

We  divide the proof of Claim 1 into several steps.

%In fact, by Theorem \ref{T:Z} and Remark \ref{R:JSalgebra} ,
%$\frac{(q_m-1)l}{q_m2^{n-m}}$ and $-\frac{l}{q_m2^{n-m}}$
%are eigenvalues of $v(0)$ for each $0\leq l\leq 2^{n-m}-1$.
%We denote
%
%
%$$
%E_l=\{i: f_i(0,0)=\frac{(q_m-1)l}{q_m2^{n-m}}, 1\leq i\leq d_n\},
%~\text{ for all}~0\leq l\leq 2^{n-m}-1
%\eqno(5.4)
%$$
%and
%$$
%E_l=\{i: f_i(0,0)=\frac{l}{q_m2^{n-m}}, 1\leq i\leq d_n\},
%~\text{ for all}~1-2^{n-m}\leq l\leq-1.
%\eqno(5.5)
%$$
%Then we have
%\begin{equation*}
%|E_l|=
%\begin{cases}
%d_mj_0,~~~l=0\\
%p_mj_l,~~~1\leq l\leq 2^{n-m}-1\\
%(d_m-p_m)j_l,~~~1-2^{n-m}\leq l\leq -1
%\end{cases},
%\end{equation*}
%and
%$$
%\{j:1\leq j\leq d_n\}=\cup_{l=1-2^{n-m}}^{2^{n-m}-1}E_l,
%$$
%where $E_{l_1}\cap E_{l_2}=\emptyset$ for any $l_1\neq l_2$.
%
%By Remark \ref{R:JSalgebra}, we have $q_n|j_{l}$,
%hence
%$$
%q_n|(|E_l|)~\text{ for each}~ 1-2^{n-m}\leq l\leq 2^{n-m}-1.\eqno(5.6)
%$$
%
%There exists some $1-2^{n-m}\leq l_0\leq 2^{n-m}-1$ such that $K\in E_{l_0}$.
%We denote
%$$
%E_{l_0}^{(1)}=\{j\in E_{l_0}:j\leq K\},
%E_{l_0}^{(2)}=\{j\in E_{l_0}:j>K\}.
%$$
%Then we have
%$$
%\{j:1\leq j\leq K\}=(\cup_{l=1-2^{n-m}}^{l_0-1}E_l)\cup E_{l_0}^{(1)}.\eqno(5.7)
%$$
%
%
%If $q_n\nmid K$, by (5.6) and (5.7), we have
%$q_n\nmid (|E_{l_0}^{(1)}|)$. In particular, we have
%$$E_{l_0}^{(1)}\neq \emptyset, E_{l_0}^{(2)}\neq \emptyset.\eqno(5.8)
%$$

Step 1. We denote
$$
s_0=\sup\{s\in[0,1]: f_K(s',0)=f_{K+1}(s',0) ~\text{for all}~s'\in[0,s]\}.
$$

First, we shall show that $s_0>0$.
Otherwise, we have $s_0=0$.

%We denote
%$$\{f_j(0,0): 1\leq j\leq d_n\}=\{e_l\}_{l=1}^N,$$
%where $e_1<e_2<\cdots<e_N$. We let $E_l=\{j:f_j(0,0)=e_l\}$
%for each $1\leq l\leq N$. There exists $1\leq l_0\leq N$
%such that $K\in E_{l_0}$. We denote
%$$
%E_{l_0}^{(1)}=\{j\in E_{l_0}:j\leq K\}, E_{l_0}^{(2)}=\{j\in E_{l_0}:j< K\}.
%$$

Since $\mu_{d_n}(0)-\mu_1(0)=\frac{2^{n-m}-1}{2^{n-m}}<1$,
$\|f_{d_n}(0,\cdot)-\mu_{d_n}(\cdot)\|<\frac{\eps}{2}$ and $\|f_1(0,\cdot)-\mu_1(0,\cdot)\|<\frac{\eps}{2}$, we have
$f_{d_n}(0,0)-f_1(0,0)<1$.
Notice that $f_j$ is continuous,
there exists $\delta_1>0$ such that
$$
f_{d_n}(s,0)-f_1(s,0)<1,~\text{for all} ~0\leq s<\delta_1.\eqno(5.4)
$$
By the definition of $s_0$, there exists $s'\in(0, \delta_1)$ such that $f_K(s',0)<f_{K+1}(s',0)$. It follows that
$$
\{f_j(s',0): 1\leq j\leq K\}\cap
\{f_j(s',0): K+1\leq j\leq d_n\}=\emptyset.
$$
By (5.4), we have
$$
\{exp(2\pi i f_j(s',0)): 1\leq j\leq K\}\cap
\{exp(2\pi i f_j(s',0)): K+1\leq j \leq d_n\}=\emptyset.
$$
Since $q_n\nmid K$, we have $v_{s'}\notin A_n$. This leads a contradiction.
Hence $s_0>0$.

\vspace{0.2in}

Step 2. Before we prove Claim 1,  we shall show the following Claim:\\
%$$
%\max_{1\leq j\leq d_n}f_j(s_0,0)-\min_{1\leq j\leq d_n}f_j(s_0,0)=1.\eqno(5.15)
%$$
%Otherwise, we have
%$$
%\max_{1\leq j\leq d_n}f_j(s_0,0)-\min_{1\leq j\leq d_n}f_j(s_0,0)<1.
%$$
%Since $f_j$ is continuous, there exists $\delta_0>0$ such that
%$$
%\max_{1\leq j\leq d_n}f_j(s,0)-\min_{1\leq j\leq d_n}f_j(s,0)<1,
%~\text{for all}~s_0<s<s_0+\delta_0.
%$$
%By the definition of $s_0$, there exists $s\in(s_0,s_0+\delta_0)$
%such that $f_K(s,0)\neq f_{K+1}(s,0)$. It follows that
%$$
%\{f_j(s,0):1\leq j\leq K\}\cap \{f_j(s,0):K+1\leq j\leq d_n\}=\emptyset.
%$$
%By (5.15), we have
%$$
%\{e^{2\pi if_j(s,0)}:1\leq j\leq K\}\cap \{e^{2\pi if_j(s,0)}:K+1\leq j\leq d_n\}=\emptyset.
%$$
%Since $q_n\nmid K$, we have $v_s\notin A_n$. A contradiction.
%So (5.15) holds.

%Further, we shall show that
{\bf Claim 2.} \hspace{0.5in} $
f_1(s_0,0)=f_K(s_0,0)
~\text{or}~
f_{d_n}(s_0,0)=f_{K+1}(s_0,0).\hspace{1.3in} (5.5)
$

If (5.5) does not hold, we have
 $f_1(s_0,0)<f_K(s_0,0)<f_{d_n}(s_0,0)$.
We denote
$$
\{f_j(s_0,0):1\leq j\leq d_n\}=\{c_k\}_{k=1}^L,
$$
where $c_1<c_2<\cdots<c_L$. Then $L\geq 3$.
We let $H_k=\{1\leq j\leq d_n:f_j(s_0,0)=c_k\}$
for $1\leq k\leq L$. In fact, $H_1=\{j: f_j(s_0,0)=f_1(s_0,0)\}$.
Then $H_{k_1}\cap H_{k_2}=\emptyset$ for $k_1\neq k_2$. Also, there exists $1<k_0<L$, such that $K\in H_{k_0}$. We let $H^{(1)}_{k_0}=\{j\in H_{k_0}:j\leq K\}$ and
$H^{(2)}_{k_0}=\{j\in H_{k_0}:j>K\}$.

Notice that $c_L-c_1\leq1$, we have $|c_{k_0}-c_k|<1$ for each $1\leq k\leq L$
and hence
$$
\{exp(2\pi i f_j(s_0,0)): j\in H_{k_0}^{(1)}\}\cap\{exp(2\pi i f_j(s_0,0)): j\in H_{k_0}^{(2)}\}=\emptyset,\eqno(5.6)
$$
and
$$
\{exp(2\pi i f_j(s_0,0)): j\in H_{k_0}^{(1)}\}\cap\{exp(2\pi i f_j(s_0,0)): j\in H_k\}=\emptyset,~\text{ for all}~k\neq k_0.\eqno(5.7)
$$

Since $v_{s_0}\in A_n$, by (5.6) and (5.7), we have
$$
q_n|(|H_{k_0}^{(1)}|).\eqno(5.8)
$$
Use a similar argument, we can prove that
$$
q_n|(|H_k|), ~\text{for all}~2\leq k\leq k_0-1. \eqno(5.9)
$$

Notice that $\{j: 1\leq j\leq K\}=(\cup_{k=1}^{k_0-1} H_k)\cup H_{k_0}^{(1)}$
and $q_n\nmid K$, we have $q_n\nmid(|H_1|)$.
That is $q_n\nmid (|\{j: f_j(s_0,0)=f_1(s_0,0)\}|)$.

We define
$$
s_1=\inf\{s\in[0,s_0]: q_n\nmid(|\{j:f_j(s,0)=f_1(s,0)\}|)\}.
$$
Notice that $q_n|(|\{j:f_j(0,0)=f_1(0,0)\}|)$ and $f_{d_n}(0,0)-f_1(0,0)<1$,
by a similar argument as the proof of (5.8), we can prove that
$s_1>0$.

We denote
$$
\{f_j(s_1,0):1\leq j\leq d_n\}=\{a_k\}_{k=1}^T,
$$
where $a_1<a_1<\cdots<a_T$. We denote $G_k=\{1\leq j\leq d_n:f_j(s_1,0)=a_k\}$ for each $1\leq k\leq T$.
In fact, $G_1=\{j:f_j(s_1,0)=f_1(s_1,0)\}$.

For any $0<\eps_1<\min\{|a_{k_1}-a_{k_2}|:k_1\neq k_2\}$, there exists $\delta_1>0$
such that
$$
|f_j(s,0)-f_j(s_1,0)|<\frac{\eps_1}{4},\eqno(5.10)
$$
for each $s$ with $0<|s-s_1|<\delta_1$

It follows that $\{f_j(s,0):j\in G_{k_1}\}\cap \{f_j(s,0):j\in G_{k_2}\}=\emptyset$ for all $k_1\neq k_2$ and $0<|s-s_1|<\delta_3$.\\

To finish the proof of Claim 2, we need to consider the following  two cases.\\

Case 1. $q_n\nmid(|\{j:f_j(s_1,0)=f_1(s_1,0)\}|)$. By the definition
of $s_1$, for any $s\in(s_1-\delta_1, s_1)$, we have
$q_n|(|\{j:f_j(s,0)=f_1(s,0)\}|)$.
We denote
$$
\{f_j(s,0):j\in G_1\}=\{r_k\}_{k=1}^W, ~\text{where} ~r_1<r_2\cdots<r_W.
$$
Further, we let $G_1^{(k)}=\{j: f_j(s,0)=r_k\}$, $k=1,2,\cdots,W$.
By (5.10), it is easy to see that $G_1^{(k)}\subset G_1$ for each $1\leq k\leq W$.
In particular,  we have $G_1^{(1)}=\{j: f_j(s,0)=f_1(s,0)\}$. Notice that
$q_n\nmid(|\{j:f_j(s_1,0)=f_1(s_1,0)\}|)$ and
$q_n|(|\{j:f_j(s,0)=f_1(s,0)\}|)$, we have $W\geq2$ and
there exists some $2\leq k_1\leq W$ such that
$q_n\nmid (|G_1^{(k_1)}|)$.

It follows that
$$
\{f_j(s,0):j\in G_1^{(k_1)}\}\cap
\{f_j(s,0):j\in G_k\}=\emptyset, ~\text{for all }~2\leq k\leq T,\eqno(5.11)
$$
and
$$
\{f_j(s,0):j\in G_1^{(k_1)}\}\cap
\{f_j(s,0):j\in G_1^{(k)}\}=\emptyset, ~\text{for all }~1\leq k\leq W ~\text{and}~k\neq k_1.
\eqno(5.12)
$$

Since $k_1\geq 2$, we have $f_j(s,0)>f_1(s,0)$ for all $j\in G_1^{(k_1)}$.
It is easy to check that
 $$
 |f_{j_1}(s,0)-f_{j_2}(s,0)|<1, j_1\in G^{(k_1)}_1, j_2\notin G_1^{(k_1)}.
$$
Combining with (5.11) and (5.12), we have
$$
\{exp(2\pi if_j(s,0)):j\in G_1^{(k_1)}\}\cap
\{exp(2\pi if_j(s,0)):j\in G_k\}=\emptyset, ~\text{for all }~2\leq k\leq T,
$$
and
$$
\{exp(2\pi if_j(s,0)):j\in G_1^{(k_1)}\}\cap
\{exp(2\pi if_j(s,0)):j\in G_1^{(k)}\}=\emptyset, ~\text{for all }~1\leq k\leq W~\text{and}~k\neq k_1.
$$
Combining with $q_n\nmid (|G_1^{(k_1)}|)$, it leads $v_{s}\notin A_n$ for any $s\in (s_1-\delta_1, s_1)$, which is a contradiction.\\

Case 2. $q_n|(|\{j:f_j(s_1,0)=f_1(s_1,0)\}|)$. Then $s_1<s_0$.
By the definition
of $s_1$, there exists $s\in(s_1, s_1+\delta_1)$ such that
$q_n\nmid(|\{j:f_j(s,0)=f_1(s,0)\}|)$.
We denote
$$
\{f_j(s,0):j\in G_1\}=\{z_k\}_{k=1}^M, ~\text{where} ~z_1<z_2\cdots<z_M.
$$
Further, we let $V^{(k)}=\{j: f_j(s,0)=z_k\}$, $k=1,2,\cdots,M$.
Also, by (5.10), we have $V^{(k)}\subset G_1$ for each $1\leq k\leq M$.
In particular, we have $V^{(1)}=\{j: f_j(s,0)=f_1(s,0)\}$. Notice that
$q_n|(|\{j:f_j(s_1,0)=f_1(s_1,0)\}|)$ and
$q_n\nmid(|\{j:f_j(s,0)=f_1(s,0)\}|)$, we have $M\geq2$ and
there exists some $2\leq k_2\leq M$ such that
$q_n\nmid (|V^{(k_2)}|)$.

It follows that
$$
\{f_j(s,0):j\in V^{(k_2)}\}\cap
\{f_j(s,0):j\in G_k\}=\emptyset, ~\text{for all }~2\leq k\leq T,\eqno(5.13)
$$
and
$$
\{f_j(s,0):j\in V^{(k_2)}\}\cap
\{f_j(s,0):j\in V^{(k)}\}=\emptyset, ~\text{for all }~1\leq k\leq W
~\text{and}~k\neq k_2.
\eqno(5.14)
$$

Since $k_2\geq 2$, we have $f_j(s,0)>f_1(s,0)$ for all $j\in V^{(k_2)}$.
It is easy to check that
 $$
 |f_{j_1}(s,0)-f_{j_2}(s,0)|<1, ~\text{for all}~j_1\in V^{(k_2)}_1, j_2\notin V^{(k_2)}.
$$
Combining with (5.13) and (5.14), we have
$$
\{exp(2\pi if_j(s,0)):j\in V^{(k_2)}\}\cap
\{exp(2\pi if_j(s,0)):j\in G_k\}=\emptyset, ~\text{for all }~2\leq k\leq T,
$$
and
$$
\{exp(2\pi if_j(s,0)):j\in V^{(k_2)}\}\cap
\{exp(2\pi if_j(s,0)):j\in V^{(k)}\}=\emptyset, ~\text{for all }~1\leq k\leq W ~ \text{and}~k\neq k_2.
$$
Combining with $q_n\nmid (|V^{(k_2)}|)$, it  leads to $v_{s}\notin A_n$ for some $s\in(s_1,s_1+\delta_1)$, which is a contradiction.

Hence (5.5) holds and Claim 2 is proved.\\

Step 3. We denote $v_s^{(1)}(t)=v_s(t)$ for $s\in[0,s_0]$
and $v_s^{(2)}(t)=v_s(t)$ for $s\in[s_0,1]$. Then we have
$$
length_s(v_s)=length_s(v_s^{(1)})+length_s(v_s^{(2)}).\eqno(5.15)
$$
If $f_1(s_0,0)=f_K(s_0,0)$, by Theorem \ref{C:increasing}, we have
$$
length_s(v_s^{(1)})\geq 2\pi |f_1(0,0)-f_1(s_0,0)|=2\pi|f_1(0,0)-f_K(s_0,0)|,\eqno(5.16)
$$
and
$$
length_s(v_s^{(2)})\geq 2\pi |f_{K+1}(s_0,0)-1|=2\pi|f_K(s_0,0)-1|.
\eqno(5.17)
$$
Combining (5.15), (5.16) and (5.17), we have
$$
length_s(v_s)\geq2\pi |f_1(0,0)-1|\geq2\pi(1+\frac{2^{n-m}-1}{q_m2^{n-m}}-\eps).\eqno(5.18)
$$

If $f_{d_n}(s_0,0)=f_{K+1}(s_0,0)$, also by Theorem \ref{C:increasing}, we have
$$
length_s(v_s^{(1)})\geq 2\pi |f_{d_n}(0,0)-f_{d_n}(s_0,0)|=2\pi|f_{d_n}(0,0)-f_{K+1}(s_0,0)|,\eqno(5.19)
$$
and
$$
length_s(v_s^{(2)})\geq 2\pi |f_{K}(s_0,0)-0|=2\pi|f_{K+1}(s_0,0)-0|.
\eqno(5.20)
$$
Combining (5.15), (5.19) and (5.20), we have
$$
length_s(v_s)\geq2\pi |f_{d_n}(0,0)-0|\geq2\pi(\frac{(q_m-1)(2^{n-m}-1)}{q_m2^{n-m}}-\eps).\eqno(5.21)
$$

Claim 1 follows from (5.18) and (5.21).\\

By a similar argument, we can prove that, if  $p_n\nmid(d_n-K)$, then
\begin{align*}
(*)\hspace{0.8in}length_s(v_s)&\geq \min\{2\pi|f_{d_n}(0,1)-0|,2\pi|f_1(0,1)-1|\}\\
&=\min\{2\pi(\frac{q_m-1}{q_m}-\eps), 2\pi(\frac{q_m+1}{q_m}-\eps)\}\\
&=2\pi(\frac{q_m-1}{q_m}-\eps).
\end{align*}
.

Now we shall show that $q_n|K$ and $p_n|(d_n-K)$ can not
hold together. Otherwise, there are positive integers $l$ and $s$
such that
$$
p_nl+q_ns=d_n.
$$
Noting that
$$
(p_n,q_n)=1,
$$
and
$$
p_n|d_n,
$$
we have $p_n|s$. We denote $s'=\frac{s}{p_n}$. It follows that
$$
l+q_ns'=\frac{d_n}{p_n}=q_n.
$$
Hence $q_n|l$. We denote $l'=\frac{l}{q_n}$. Then
$$
l'+s'=1.
$$
This contradicts to the fact that $l'$ and $s'$ are positive integers.

Hence either $q_n\nmid K$ or $p_n\nmid (d_n-K)$.

As $m$ goes to infinity and $\eps$ goes to zero, applying Claim 1 (for the case $q_n\nmid K$) and $(*)$ (for the case $p_n\nmid (d_n-K)$), we get an element $v:=\phi_{m, \infty} (u)\in CU(\mathcal{Z})$ such  that $cel(v)\geq 2\pi-\eta$, for any pregiven positive number $\eta$. Hence
we have
$$
cel_{CU}(\mathcal{Z})\geq 2\pi.
$$
\end{proof}

Further, we get a more general result.

\begin{thm}\label{T:matrixJiangSu}
Let $\mathcal{Z}$ be the Jiang-Su algebra and $k$ be a positive integer.
Then
$$
cel_{CU}(M_k(\mathcal{Z}))\geq 2\pi.
$$	
\end{thm}

\begin{proof}
Let $\mathcal{Z}=\lim_{m}A_m$ be the Jiang-Su algebra.
For each $m\geq1$, we define a unitary $u_1\in A_m$
as  follows:
$$
u_1(t)=
\begin{bmatrix}
e^{2\pi i h_1(t)}&&\\
&\ddots&\\
&&e^{2\pi ih_{d_m}(t)}	
\end{bmatrix}_{d_m\times d_m},
$$	
where $h_i(t)=\frac{q_m-1}{q_m}t$ for each $1\leq i\leq p_m$ and
$h_i(t)=-\frac{1}{q_m}t$ for each $p_m+1\leq i\leq d_m$.
We denote
$$
u=diag[u_1,u_2,\cdots,u_k]\in M_k(A_m),
$$
where $u_i(t)={\bf 1}_{A_m}$ for each $2\leq i\leq k$.
It follows that $u\in CU(M_k(A_m))$.

For any $\eps>0$, use a similar proof of Theorem \ref{T:JiangSu}, we can prove that
$$
cel(u)\geq 2\pi-\eps.
$$

\end{proof}

\bibliographystyle{amsplain}
%\bibliography{ref}

\end{document}